\documentclass[10pt]{amsart}
\usepackage[margin=1in,letterpaper,portrait]{geometry}
\usepackage{ytableau,shuffle}
\usepackage[enableskew,vcentermath]{youngtab}
\usepackage{latexsym,amsmath,amssymb,verbatim,tikz}
\usepackage[pdfencoding=auto, psdextra]{hyperref}
\usepackage{graphicx,xcolor,xypic,accents}

\makeatletter
\DeclareRobustCommand*{\bfseries}{%
	\not@math@alphabet\bfseries\mathbf
	\fontseries\bfdefault\selectfont
	\boldmath
}
\makeatother

\theoremstyle{plain}

\newtheorem{theorem}{Theorem}[section]
\newtheorem{proposition}[theorem]{Proposition}

\newtheorem{lemma}[theorem]{Lemma}
\newtheorem{corollary}[theorem]{Corollary}
\newtheorem{conjecture}[theorem]{Conjecture}
\newtheorem{problem}[theorem]{Problem}
\theoremstyle{definition}
\newtheorem{definition}[theorem]{Definition}
\newtheorem{defn}[theorem]{Definition}
\newtheorem{example}[theorem]{Example}

\newtheorem{remark}[theorem]{Remark}
\newtheorem{observation}[theorem]{Observation}
\newtheorem{obs}[theorem]{Observation}

\numberwithin{equation}{section}

\newcommand{\todo}[1]{\vspace{5 mm}\par \noindent
	\marginpar{\textsc{ToDo}} \framebox{\begin{minipage}[c]{0.95
				\textwidth}
			#1 \end{minipage}}\vspace{5 mm}\par}

\newcommand{\PP}{{\mathbb {P}}}

\newcommand{\QQ}{{\mathbb {Q}}}

\newcommand{\ZZ}{{\mathbb {Z}}}

\newcommand{\then}{\Rightarrow}

\newcommand{\spn}{{\operatorname{span}}}

\newcommand{\Des}{{\operatorname{Des}}}
\newcommand{\cDes}{{\operatorname{cDes}}}
\newcommand{\cyc}{{\operatorname{cyc}}}
\newcommand{\cdes}{{\operatorname{cdes}}}
\newcommand{\des}{{\operatorname{des}}}

\newcommand{\ch}{{\operatorname{ch}}}

\newcommand{\SYT}{{\operatorname{SYT}}}
\newcommand{\bded}{{\operatorname{bd}}}
\newcommand{\co}{{\operatorname{co}_n}}
\newcommand{\cc}{{\operatorname{cc}_n}}

\newcommand{\Comp}{{\operatorname{Comp}}}
\newcommand{\cComp}{{\operatorname{cComp}}}
\newcommand{\cnes}{{c2_0^{[n]}}}

\DeclareMathOperator{\QSym}{QSym}
\DeclareMathOperator{\Sym}{Sym}
\DeclareMathOperator{\cQSym}{cQSym}

\newcommand{\cs}{{\tilde{s}}}
\newcommand{\hM}{{\widehat{M}}}
\newcommand{\hF}{{\widehat{F}}}

\newcommand{\xx}{{\mathbf{x}}}
\newcommand{\ttt}{{\mathbf{t}}}
\newcommand{\symm}{{\mathfrak{S}}}

\newcommand{\AAA}{{\mathcal{A}}}
\newcommand{\BBB}{{\mathcal{B}}}

\newcommand{\LLL}{{\mathcal{L}}}

\newcommand{\TTT}{{\mathcal{T}}}
\newcommand{\Q}{{\mathcal Q}}

\def\tor{{\operatorname{tor}}}

\newcommand{\DAG}[1]{\accentset{\rightharpoonup}{#1}}

\begin{document}
	
	\title[Cyclic quasi-symmetric functions]
	{Cyclic quasi-symmetric functions}
	\author{Ron M.\ Adin}
	\address{Department of Mathematics, Bar-Ilan University, 
		Ramat-Gan 52900, Israel}
	\email{radin@math.biu.ac.il}
	\author{Ira M.\ Gessel}
	\address{Department of Mathematics, Brandeis University,
		Waltham, MA 02453, USA}
	\email{gessel@brandeis.edu}
	\author{Victor Reiner}
	\address{School of Mathematics, University of Minnesota,
		Minneapolis, MN 55455, USA}
	\email{reiner@math.umn.edu}
	\author{Yuval Roichman}
	\address{Department of Mathematics, Bar-Ilan University, 
		Ramat-Gan 52900, Israel}
	\email{yuvalr@math.biu.ac.il}
	
	\date{May 25, 2020}
	
	\thanks{RMA and YR were partially supported by an MIT-Israel MISTI grant and by the Israel Science Foundation, grant no.\ 1970/18. 
		IMG was partially supported by grant no.\ 427060 from the Simons Foundation.
		VR was partially supported by NSF grant DMS-1601961.}
	
	\begin{abstract}
		The ring of cyclic quasi-symmetric functions 
		and its non-Escher subring 
		are introduced in this paper. 
		A natural basis consists of
		fundamental cyclic quasi-symmetric functions;
		for the non-Escher subring they arise as toric $P$-partition enumerators, for toric posets $P$ with a total cyclic order.
		The associated structure constants are determined by cyclic shuffles of permutations. 
		We then prove the following positivity phenomenon: 
		for every non-hook shape $\lambda$, the coefficients in the expansion of the Schur function $s_\lambda$ in terms of fundamental cyclic quasi-symmetric functions are nonnegative. 
		The proof relies on the existence of a cyclic descent map on the standard Young tableaux (SYT) of shape $\lambda$. 
		The theory has applications to  
		the enumeration of cyclic shuffles and SYT
		by cyclic descents. 
	\end{abstract}
	
	
	
	\maketitle
	
	\tableofcontents
	
	\section{Introduction}
	
	The graded rings $\Sym$ (sometimes denoted $\Lambda$)
	and $\QSym$,
	of symmetric and quasi-symmetric functions, respectively,
	have many applications to enumerative combinatorics, as well as to other branches of mathematics; see, e.g., 
	\cite[Ch.\ 7]{EC2}.
	This paper introduces two intermediate objects: the graded ring $\cQSym$ of cyclic quasi-symmetric functions,
	and its non-Escher subring $\cQSym^-$.
	
	\smallskip
	
	The rings $\Sym$, $\QSym$ and $\cQSym$ may be defined via invariance properties. 
	A formal power series $f\in \ZZ[[x_1,x_2,\ldots]]$ of bounded degree is {\em symmetric} if 
	for any $t \ge 1$, any two sequences $(i_1, \ldots, i_t)$
	and $(j_1, \dots, j_t)$ of distinct positive integers (indices), 
	and any sequence $(m_1, \ldots, m_t)$ of positive integers (exponents), 
	the coefficients of $x_{i_1}^{m_1} \cdots x_{i_t}^{m_t}$ and 
	$x_{j_1}^{m_1} \cdots x_{j_t}^{m_t}$ in $f$ are equal. 
	We call $f$ {\em quasi-symmetric} if 
	for any $t \ge 1$, any two {\em increasing} sequences $i_1 < \dots < i_t$ and $j_1 < \dots < j_t$ of positive integers, 
	and any sequence $(m_1, \ldots, m_t)$ of positive integers, 
	the coefficients of $x_{i_1}^{m_1} \cdots x_{i_t}^{m_t}$ and 
	$x_{j_1}^{m_1} \cdots x_{j_t}^{m_t}$ in $f$ are equal. 
	%
	%
	We define $f$ to be
	{\em cyclic quasi-symmetric} if 
	for any $t \ge 1$, any two increasing sequences $i_1 < \dots < i_t$ and $j_1 < \dots < j_t$ of positive integers, any sequence $m = (m_1, \ldots, m_t)$ of positive integers, and any {\em cyclic shift} $m' = (m'_1, \ldots, m'_t)$ of $m$, 
	the coefficients of
	$x_{i_1}^{m_1} \cdots x_{i_t}^{m_t}$ and 
	$x_{j_1}^{m'_1} \cdots x_{j_t}^{m'_t}$ in $f$ are equal. 
	In fact, the set $\cQSym$ of cyclic quasi-symmetric functions  is a graded ring satisfying
	\[
	\Sym \subseteq \cQSym \subseteq \QSym.
	\]
	We remark that all algebras (and coalgebras) in this paper are defined over $\ZZ$, so that all structure constants are (mostly nonnegative) integers.
	Consequently, they may also be defined over an arbitrary field.
	
	\medskip
	
	Toric posets were recently introduced by Develin, Macauley and Reiner~\cite{DMR}. 
	A toric analogue of $P$-partitions is presented in Section~\ref{sec:toric}.
	Toric $P$-partition enumerators, in the special case of total cyclic orders, 
	form a convenient $\QQ$-basis for a ring $\cQSym^-$, which is a subring of $\cQSym$.
	A slightly extended set actually forms a $\QQ$-basis for $\cQSym$ itself.
	The elements of this basis are called 
	{\em fundamental cyclic quasi-symmetric functions},
	are indexed by cyclic compositions of a positive integer $n$
	(equivalently, by cyclic equivalence classes of nonempty subsets $J \subseteq [n]:=\{1,2,\ldots,n\}$), 
	and are denoted $F^\cyc_{n,[J]}$.
	Normalized versions of them form $\ZZ$-bases for $\cQSym$ and $\cQSym^-$; 
	see Lemma~\ref{t:Fcyc_basis} and Observation~\ref{t:proper_basis} below.
	The full ring $\cQSym$ contains, 
	in particular, all symmetric functions.
	Its subring $\cQSym^-$ 
	forms a natural framework for the study of cyclic descent sets of combinatorial objects,
	possessing the non-Escher property;
	see
	Definition~\ref{def:cDes} and Remark~\ref{rem:non-Escher}.
	
	\medskip
	
	A toric analogue of Stanley's fundamental decomposition lemma for $P$-partitions~\cite[Lemma 3.15.3]{EC1}, given in Lemma~\ref{cyclic-fundamental-lemma} below, 
	is applied to provide a combinatorial interpretation of the resulting structure constants in terms of shuffles of cyclic permutations (more accurately: cyclic words), as follows.
	
	
	
	For a finite set $A$ of size $a$, let $\symm_A$ be 
	the set of all bijections $u : [a] \to A$,
	viewed as words $u = (u_1, \ldots, u_a)$. 
	Elements of $\symm_A$ will be called {\em bijective words}, or simply {\em words}.  Call $A$ the {\it support} of $u$.
	If $A = [a]$ then $\symm_A$ is the symmetric group $\symm_a$, whose elements are genuine permutations.
	
	\medskip
	

	If $A$ is a finite set of integers, or any finite totally ordered set, define the {\em cyclic descent set} of $u \in \symm_A$  by
	\begin{equation}
		\cDes(u) := \{1 \leq i \leq a \,:\, u_i > u_{i+1} \}
		\quad \subseteq [a],
	\end{equation}
	with the convention $u_{a+1} := u_1$.
	The {\em cyclic descent number} of $u$ is $\cdes(u) := |\cDes(u)|$.
	
	The cyclic descent set was introduced by Cellini
	in the study of Solomon's descent algebra~\cite{Cellini95, Cellini}.
	For enumerative and other
	aspects of this statistic 
	see, e.g.,~\cite{Fulman,DPS, Petersen, AP, BER}.
	Generalizations of cyclic descents to 
	standard Young tableaux and other combinatorial objects
	involve surprising algebraic, combinatorial and topological connections~\cite{Rhoades, ARR_CDes,
		Huang, HR}.
	
	\medskip
	
	A {\em cyclic 
		word} $[\DAG u]\in \symm_A/\ZZ_a$ is an equivalence class of elements of $\symm_A$ under the cyclic equivalence relation 
	$(u_1,\ldots,u_a) \sim (u_{i+1}\ldots,u_a,u_1,\ldots,u_i)$ for all $i$.
	A {\em cyclic shuffle}  of two cyclic 
	words $[\DAG u]$ and $[\DAG v]$ with disjoint supports
	is any cyclic equivalence class $[\DAG w]$ represented by a shuffle $w$ of a representative of $[\DAG u]$ and a representative of $[\DAG v]$.
	The set of all cyclic shuffles of $[\DAG u]$ and $[\DAG v]$ is denoted $[\DAG u] \shuffle_{\cyc} [\DAG v]$,
	and is a collection of cyclic equivalence classes.
	
	
	The following cyclic analogue of 
	the product formula for (ordinary) fundamental quasi-symmetric functions~\cite[Ex.\ 7.93]{EC2} provides a combinatorial interpretation for the structure constants of $\cQSym^-$.
	
	\begin{theorem}\label{thm:main1}
		Let $C = A \sqcup B$ be a disjoint union of finite sets of integers.
		For each $u \in \symm_A$ and $v \in \symm_B$, one has 
		the following expansion:
		\[
		F^\cyc_{|A|,\cDes(u)} \cdot F^\cyc_{|B|,\cDes(v)}
		= \sum_{[\DAG{w}] \in [\DAG{u}] \shuffle_{\cyc} [\DAG{v}] } F^\cyc_{|C|,\cDes(w)}.
		\]
	\end{theorem}
	
	See Theorem~\ref{product-of-F-corollary} below.
	
	
	
	\bigskip
	
	
	Recall that a skew shape is called a {\em ribbon} if it does not contain a $2\times  2$ square.
	The following positivity phenomenon is proved in Section~\ref{sec:expansion}.
	
	\begin{theorem}\label{thm:main2}
		For every skew shape $\lambda/\mu$ which is not a connected ribbon, 
		all the coefficients in the expansion of the skew Schur function $s_{\lambda/\mu}$
		in terms of normalized fundamental cyclic quasi-symmetric functions are nonnegative integers. 
	\end{theorem}
	
	A more precise statement, which provides a combinatorial interpretation for the coefficients, is given in  
	Corollary~\ref{t:Schur_in_hFcyc} below.
	The proof relies on the existence of a cyclic extension of the descent map on standard Young tableaux (SYT) of shape $\lambda/\mu$, which was proved in~\cite{ARR_CDes}. 
	Using 
	Postnikov's result regarding toric Schur functions, 
	one deduces that the coefficients  in the expansion of a non-hook Schur function $s_\lambda$
	in terms of fundamental cyclic quasi-symmetric functions
	are equal to certain Gromov-Witten invariants, see Remark~\ref{rem_Postnikov} below. 
	
	\bigskip
	
	
	Applications to  
	the enumeration of SYT and cyclic shuffles of permutations with prescribed cyclic descent set or prescribed number of cyclic descents follow from this theory. Using a ring homomorphism from $\cQSym$ to the ring of formal power series $\ZZ[[q]]_{\odot}$, with product defined by $q^i \odot q^j := q^{\max(i,j)}$,
	Theorem~\ref{thm:main1} implies the following cyclic analogue 
	of a classical result of Stanley~\cite[Prop.\ 12.6]{Stanley_memoir}.
	
	\begin{theorem}
		Let $A$ and $B$ be two disjoint sets of integers,
		with $|A|=m$ and $|B|=n$.
		For each  $u \in \symm_A$ and $v \in \symm_B$,
		if $\cdes(u) = i$ and $\cdes(v) = j$ then the number of cyclic shuffles $[\DAG w]$ in $[\DAG{u}] \shuffle_{\cyc} [\DAG{v}]$ with $\cdes(w)=k$ is equal to
		\[
		\frac{k(m-i)(n-j) + (m+n-k)ij}{(m+j-i)(n+i-j)} \binom{m+j-i}{k-i} \binom{n+i-j}{k-j}. 
		\]
	\end{theorem}
	
	See Corollary~\ref{t:num_shuffles_given_cdes} below.
	More enumerative applications are given in Section~\ref{sec:enumeration}.
	
	\bigskip
	
	
	The group ring $\ZZ[\symm_n]$ has a distinguished subring, {\em Solomon's descent algebra} $\mathfrak{D}_n$~\cite{Solomon}, with basis elements
	\[
	D_I 
	:= \sum_{\substack{\pi \in \symm_n \\ \Des(\pi) = I}} \pi
	\qquad (I \subseteq [n-1]).
	\]
	An operation dual to the product in $\mathfrak{D}_n$ is the {\em internal coproduct} $\Delta_n$ on $\QSym_n$; 
	see Section~\ref{sec:co} for its definition.
	
	\begin{theorem}
		$\cQSym_n$ and $\cQSym_n^-$ are right comodules of $\QSym_n$ with respect to the internal coproduct:
		\[
		\Delta_n(\cQSym_n) \subseteq \cQSym_n \otimes \QSym_n 
		\]
		and
		\[
		\Delta_n(\cQSym_n^-) \subseteq \cQSym_n^- \otimes \QSym_n. 
		\]
		The structure constants for $\cQSym_n^-$ are nonnegative integers.
	\end{theorem}
	For a more detailed description, including a combinatorial interpretation of the coefficients, see Theorem~\ref{t:cQSym_right_comodule}.
	This reproves the following result, obtained by Moszkowski~\cite{Mo} and by Aguiar and Petersen~\cite{AP} using different methods.
	
	\begin{corollary}
		For $n > 1$
		let $c2_{0,n}^{[n]}$ be the set of equivalence classes, under cyclic rotations, of subsets $\varnothing \subsetneq J \subsetneq [n]$.
		Defining
		\[
		cD_A
		:= \sum_{\substack{\pi \in \symm_n \\ \cDes(\pi) \in A}} \pi
		\qquad (A \in c2_{0,n}^{[n]}),
		\]
		the additive free abelian group
		\[
		\mathfrak{cD}_n := \spn_{\,\ZZ} \{cD_A \,:\, A \in c2_{0,n}^{[n]}\}
		\]
		is a left module for Solomon's descent algebra $\mathfrak{D}_n$.
	\end{corollary}
	In general, $\mathfrak{cD_n}$ 
	is not an algebra, namely, is not closed under the multiplication in $\ZZ[\symm_n]$.
	
	
	
	\bigskip
	
	We conclude the paper with final remarks and open problems.

	\section{The ring of cyclic quasi-symmetric functions}
	
	Recall from \cite{Gessel} the following basic definitions:
	
	A {\em quasi-symmetric function} is
	a formal power series $f \in \ZZ[[x_1, x_2, \ldots]]$  of bounded degree such that,
	for any $t \ge 1$, any two increasing sequences $i_1 < \dots < i_t$
	and $i'_1 < \dots < i'_t$ of positive integers, and any sequence
	$(m_1, \ldots, m_t)$ of positive integers, the coefficients of
	$x_{i_1}^{m_1} \cdots x_{i_t}^{m_t}$ and 
	$x_{i'_1}^{m_1} \cdots x_{i'_t}^{m_t}$ in $f$ are equal.
	Denote by $\QSym$ the set of all quasi-symmetric functions, 
	and by $\QSym_n$ the set of all quasi-symmetric functions
	which are homogeneous of degree $n$.
	
	For any positive integer $n$, 
	there is a natural bijection $\co: 2^{[n-1]} \to \Comp_n$, 
	from the power set of $[n-1] = \{1, \ldots, n-1\}$ 
	to the set of all compositions of $n$,
	defined by 
	\[
	J = \{j_1 < \dots < j_t\} \subseteq [n-1] \,\then\, 
	\co(J) := (j_1, j_2-j_1, \ldots, j_t-j_{t-1}, n-j_t). 
	\]
	In particular, $\co(\varnothing) = (n)$.
	
	
	The {\em monomial quasi-symmetric function} 
	corresponding to a subset $J \subseteq [n-1]$ is defined by
	\[
	M_{n,J} := \sum x_{i_1} \cdots x_{i_n},
	\]
	where the sum extends over all sequences $(i_1, \ldots, i_n)$
	of positive integers such that $j \in J \then i_j < i_{j+1}$ and
	$j \not\in J \then i_j = i_{j+1}$.
	Equivalently, letting $\co(J) = (m_0, \ldots, m_t)$ 
	where $t=\#J$, 
	\[
	M_{n,J} = M_{(m_0,\ldots,m_t)}
	:= \sum_{i_0 < \dots < i_t} x_{i_0}^{m_0} \cdots x_{i_t}^{m_t}. 
	\]
	The set $\{M_{n,J}:\ J\subseteq [n-1]\}$ forms a basis for the additive abelian group $\QSym_n$.
	
	Similarly, the {\em fundamental quasi-symmetric function}
	corresponding to a subset $J \subseteq [n-1]$ is defined by
	\[
	F_{n,J} := \sum x_{i_1} \cdots x_{i_n},
	\]
	where the sum extends over all sequences $(i_1, \ldots, i_n)$
	of positive integers such that $j \in J \then i_j < i_{j+1}$ and
	$j \not\in J \then i_j \le i_{j+1}$.
	For all $J \subseteq [n-1]$ we have
	\[
	F_{n,J} = \sum_{K \supseteq J} M_{n,K} 
	\]
	and consequently, by inclusion-exclusion,
	\[
	M_{n,J} = \sum_{K \supseteq J} (-1)^{\#(K \setminus J)} F_{n,K}.
	\]
	Thus, the set $\{F_{n,J}:\ J\subseteq [n-1]\}$ forms a basis for the additive abelian group $\QSym_n$.
	
	\bigskip
	
	We shall now define cyclic analogues of these concepts.
	
	\begin{defn}\label{def:cQSym}
		A {\em cyclic quasi-symmetric function} is
		a formal power series $f \in \ZZ[[x_1, x_2, \ldots]]$ of bounded degree such that,
		for any $t \ge 1$, any two increasing sequences $i_1 < \dots < i_t$
		and $i'_1 < \dots < i'_t$ of positive integers, any sequence
		$m = (m_1, \ldots, m_t)$ of positive integers, and any {\em cyclic shift}
		$m' = (m'_1, \ldots, m'_t)$ of $m$, the coefficients of
		$x_{i_1}^{m_1} \cdots x_{i_t}^{m_t}$ and 
		$x_{i'_1}^{m'_1} \cdots x_{i'_t}^{m'_t}$ in $f$ are equal. 
	\end{defn}
	
	Denote by $\cQSym$ the set of all cyclic quasi-symmetric functions, 
	and by $\cQSym_n$ the set of all cyclic quasi-symmetric functions
	which are homogeneous of degree $n$.
	
	\begin{obs}\label{t:QSym_subring}
		$\QSym$, $\cQSym$ and the set $\Sym$ of symmetric functions 
		(sometimes denoted $\Lambda$) are graded abelian groups satisfying
		\begin{equation}
			\label{ring-inclusions1}
			\Sym \subseteq \cQSym \subseteq \QSym.
		\end{equation}
		It is not too difficult to check that they are also rings, that is, closed under multiplication of formal power series.  
		For $\Sym$ and $\QSym$ this is well-known.
		The proof for $\cQSym$,
		with a combinatorial interpretation of the structure constants,
		is deferred to the end of Section~\ref{sec:toric_posets};
		see Proposition~\ref{t:cQSym_is_a_ring}.
	\end{obs}
	

	\subsection{Monomial cyclic quasi-symmetric functions}
	
	For any positive integer $n$, 
	let $2^{[n]}$ be the set of all subsets of $[n]$,
	and let $2_0^{[n]}$ be the set of all {\em nonempty} subsets of $[n]$.
	There is a natural map $\cc: 2_0^{[n]} \to \Comp_n$, 
	defined by 
	\[
	J = \{j_1 < \dots < j_t\} \subseteq [n] \,\mapsto\, 
	\cc(J) := (j_2-j_1, \ldots, j_t-j_{t-1}, j_1-j_t+n). 
	\]
	In particular, $\cc(\{j\}) = (n)$ for any $j \in [n]$,
	while $\cc(\varnothing)$ is undefined.
	
	A {\em cyclic shift} of a subset $J \subseteq [n]$ is a set of the form
	$k+J := \{k+j \pmod n \,:\, j \in J\}$ for some $k \in [n]$.
	Clearly, if $J'$ is a cyclic shift of $J$ then 
	$\cc(J')$ is a cyclic shift of $\cc(J)$.
	The converse is also true---in fact, if $\cnes$ 
	(respectively, $\cComp_n$)
	denotes the set of equivalence classes of elements of $2_0^{[n]}$
	(respectively, $\Comp_n$) under cyclic shifts, then the induced map
	$\cc: \cnes \to \cComp_n$ 
	is a bijection.
	The elements of $\cComp_n$ are called {\em cyclic compositions} of $n$.
	
	Note that $\cnes$ is the set of orbits of $2_0^{[n]}$ under the natural action of the cyclic group $\ZZ/n\ZZ$. On the other hand, $\Comp_n$ contains compositions of varying lengths $1 \le t \le n$, so that $\cComp_n$ consists of orbits of the corresponding groups $\ZZ/t\ZZ$.
	
	\begin{remark}
		Burnside's Lemma, applied to the natural action of the cyclic group $\ZZ/n\ZZ$ on $2^{[n]}$, implies that
		\[
		\#\cnes 
		= \frac{1}{n} \sum_{d | n} \varphi(d) (2^{n/d} - 1)
		= \left( \frac{1}{n} \sum_{d | n} \varphi(d) 2^{n/d} \right) - 1,
		\]
		where $\varphi$ is Euler's totient function.
	\end{remark}
	
	\begin{defn}\label{def:Mcyc}
		For a nonempty subset $J = \{j_1 < \dots < j_t\} \in 2_0^{[n]}$,
		let $\cc(J) = (m_1, \ldots, m_t)$.
		The {\em monomial cyclic quasi-symmetric function}
		corresponding to $J$ is defined by
		\[
		M^{\cyc}_{n,J} 
		:= \sum_{k = 1}^{t} M_{(m_k,m_{k+1},\ldots,m_{k+t-1})}
		= \sum_{k = 1}^{t} \,\,  \sum_{i_1 < \dots < i_t}
		x_{i_1}^{m_{k}} \cdots x_{i_t}^{m_{k+t-1}},
		\]
		where indices are added cyclically in $[t]$, meaning $m_{j+t}=m_j$.
		Define also $M^{\cyc}_{n,\varnothing} := 0$.
	\end{defn}
	
	\begin{lemma}\label{t:Mcyc_from_M}
		{\rm (From monomial to cyclic monomial)}
		For every subset $J \subseteq [n]$
		\[
		M^{\cyc}_{n,J} = \sum_{j \in J} M_{n, (J - j) \cap [n-1]},
		\]
		where $J - j := \{k - j \,:\, k \in J\} \subseteq [n]$
		with subtraction interpreted cyclically modulo $n$.
	\end{lemma}
	
	\begin{proof}
		For $J = \varnothing$ both sides are zero.
		For $J = \{j_1 < \dots < j_t\}$ nonempty with $\cc(J) = (m_1, \ldots, m_t)$,
		the claim follows directly from Definition~\ref{def:Mcyc} since, for each $k \in [t]$,
		\[
		M_{n, (J - j_k) \cap [n-1]} 
		= M_{n, \{j_{k+1} - j_k, \ldots, j_t - j_k, j_1 - j_k +n, \ldots, j_{k-1} - j_k + n\}}  
		= \sum_{i_1 < \dots < i_t} x_{i_1}^{m_{k}} \cdots x_{i_t}^{m_{k+t-1}}.
		\qedhere
		\]
	\end{proof}

	\begin{observation}\label{t:Mcyc_cyclic_invariance}
		{\rm (Cyclic invariance)}
		If $J' \in 2^{[n]}$ is a cyclic shift of $J \in 2^{[n]}$ then 
		$M^{\cyc}_{n,J'} = M^{\cyc}_{n,J}$.
	\end{observation}
	
	\begin{example}
		We illustrate Definition~\ref{def:Mcyc} , Lemma~\ref{t:Mcyc_from_M}, and Observation~\ref{t:Mcyc_cyclic_invariance} by computing some $M^{\cyc}_{n,J}$.
		
		\begin{center}
			\begin{tabular}{|c|c|}
				\hline
				$J \subseteq [2]$ & $M^{\cyc}_{2,J}$ \\ 
				\hline\hline
				$\varnothing$ & $0$ \\ 
				\hline
				$\{1\}$ or $\{2\}$ & $M_{2,\varnothing}=M_{(2)}$ \\ 
				\hline
				$\{1,2\}$ & $2M_{2,\{1\}}=2M_{(1,1)}$ \\ 
				\hline
			\end{tabular}
		\end{center}
		
		\bigskip
		
		\begin{center}
			\begin{tabular}{|c|c|}
				\hline 
				$J \subseteq [3]$ & $M^{\cyc}_{3,J}$ \\ 
				\hline\hline
				$\varnothing$ & $0$ \\ 
				\hline
				$\{1\}$ or $\{2\}$ or $\{3\}$ & $M_{3,\varnothing}=M_{(3)}$ \\ 
				\hline
				$\{1,2\}$ or $\{2,3\}$ or $\{1,3\}$ &   $M_{3,\{1\}}+M_{3,\{2\}}=M_{(1,2)}+M_{(2,1)}$ \\
				\hline
				$\{1,2,3\}$ & $3M_{3,\{1,2\}}=3M_{(1,1,1)}$ \\
				\hline
			\end{tabular}
		\end{center}
		
		\bigskip
		
		\begin{center}
			\begin{tabular}{|c|c|}
				\hline
				$J \subseteq [4]$ & $M^{\cyc}_{4,J}$ \\ 
				\hline\hline
				$\varnothing$ & $0$ \\ 
				\hline
				$\{1\}$ or $\{2\}$ or $\{3\}$ or $\{4\}$ & $M_{4,\varnothing}=M_{(4)}$ \\
				\hline
				$\{1,2\}$ or $\{2,3\}$ or $\{3,4\}$ or $\{1,4\}$ & $M_{4,\{1\}}+M_{4,\{3\}}=M_{(1,3)}+M_{(3,1)}$ \\
				\hline
				$\{1,3\}$ or $\{2,4\}$ & $2M_{4,\{2\}}=2M_{(2,2)}$ \\
				\hline
				$\{1,2,3\}$ & 
				$M_{4,\{1,2\}}+M_{4,\{1,3\}}+M_{4,\{2,3\}}=M_{(1,1,2)}+M_{(1,2,1)}+M_{(2,1,1)}$ \\
				\hline
				$\{1,2,3,4\}$ & $4M_{4,\{1,2,3\}}=4M_{(1,1,1,1)}$ \\
				\hline
			\end{tabular}
		\end{center}
		
	\end{example}

	\begin{example}
		In the previous example one finds that each $M^{\cyc}_{n,J}$ for $n=2,3,4$ lies not only in
		$\cQSym$, but also in the subring of symmetric functions $\Sym$.  
		This reflects the fact that the inclusion of graded rings $\Sym \subseteq \cQSym$
		is an isomorphism up through degree $5$, and only becomes a proper inclusion in degree $6$.  
		To see an example of this proper inclusion, consider the following table showing $M^{\cyc}_{6,J}$ for one
		representative of each cyclic (modulo $6$) class of subsets $J \subseteq [6]$:
		
		\begin{center}
			\begin{tabular}{|c|c|}
				\hline
				$J \subseteq [6]$ & $M^{\cyc}_{6,J}$ \\ 
				\hline\hline
				$\varnothing$ & $0$ \\ 
				\hline
				$\{1\}$  & $M_{6,\varnothing}=M_{(6)}$ \\
				\hline
				$\{1,2\}$ & $M_{6,\{1\}}+M_{6,\{5\}}=M_{(1,5)}+M_{(5,1)}$ \\
				\hline
				$\{1,3\}$ & $M_{6,\{2\}}+M_{6,\{4\}}=M_{(2,4)}+M_{(4,2)}$ \\
				\hline
				$\{1,4\}$ & $2M_{6,\{3\}}=2M_{(3,3)}$ \\
				\hline
				$\{1,2,3\}$& 
				$M_{6,\{1,2\}}+M_{6,\{1,5\}}+M_{6,\{4,5\}}=M_{(1,1,4)}+M_{(1,4,1)}+M_{(4,1,1)}$ \\
				\hline
				$\{1,2,4\}$& 
				$M_{6,\{1,3\}}+M_{6,\{2,5\}}+M_{6,\{3,4\}}=M_{(1,2,3)}+M_{(2,3,1)}+M_{(3,1,2)}$ \\
				\hline
				$\{1,2,5\}$& 
				$M_{6,\{1,4\}}+M_{6,\{3,5\}}+M_{6,\{2,3\}}=M_{(1,3,2)}+M_{(3,2,1)}+M_{(2,1,3)}$ \\
				\hline
				$\{1,3,5\}$& 
				$3M_{6,\{2,4\}}=3M_{(2,2,2)}$ \\
				\hline
				$\{1,2,3,4\}$ & $M_{6,\{1,2,3\}}+M_{6,\{1,2,5\}}+M_{6,\{1,4,5\}}+M_{6,\{3,4,5\}}$\\
				&$=M_{(1,1,1,3)}+M_{(1,1,3,1)}+M_{(1,3,1,1)}+M_{(3,1,1,1)}$ \\
				\hline
				$\{1,2,3,5\}$ & $M_{6,\{1,2,4\}}+M_{6,\{1,3,5\}}+M_{6,\{2,4,5\}}+M_{6,\{2,3,4\}}$\\
				&$=M_{(1,1,2,2)}+M_{(1,2,2,1)}+M_{(2,2,1,1)}+M_{(2,1,1,2)}$ \\
				\hline
				$\{1,2,4,5\}$ & $2M_{6,\{1,3,4\}}+2M_{6,\{2,3,5\}}$\\
				&$=2M_{(1,2,1,2)}+2M_{(2,1,2,1)}$ \\
				\hline
				$\{1,2,3,4,5\}$ & $M_{6,\{1,2,3,4\}}+M_{6,\{1,2,3,5\}}+M_{6,\{1,2,4,5\}}+M_{6,\{1,3,4,5\}}+M_{6,\{2,3,4,5\}}$\\
				&  $=M_{(1,1,1,1,2)}+M_{(1,1,1,2,1)}+M_{(1,1,2,1,1)}+M_{(1,2,1,1,1)}+M_{(2,1,1,1,1)}$ \\
				\hline
				$\{1,2,3,4,5,6\}$ & $6M_{6,\{1,2,3,4,5,6\}}=6M_{(1,1,1,1,1,1)}$ \\
				\hline
			\end{tabular}
		\end{center}
		
		\noindent
		Here, neither of $M_{6,\{1,2,4\}}$ and $M_{6,\{1,2,5\}}$ lies in $\Sym$,
		but their sum $M_{6,\{1,2,4\}}+M_{6,\{1,2,5\}}$ does.
		Similarly, neither of $M_{6,\{1,2,3,5\}}$ and $M_{6,\{1,2,4,5\}}$ lies in $\Sym$, but  $2M_{6,\{1,2,3,5\}}+M_{6,\{1,2,4,5\}}$ does.
		
	\end{example}
	

	As we shall see, a suitable set of functions $M^{\cyc}_{n,J}$ will form a basis for the vector space $\cQSym_n\! \otimes \,\QQ$.
	In order to get a $\ZZ$-basis for $\cQSym_n$, they need to be normalized; see Subsection~\ref{sec:normalization}.

	\subsection{Fundamental cyclic quasi-symmetric functions}\label{sec:FCQSym}
	
	\begin{definition}\label{def:cyclic_word}
		Let $\PP:=\{1,2,3,\ldots\}$.  
		For $n \ge 1$ and a
		subset $J \subseteq [n]$, denote by $P^{\cyc}_{n,J}$ the set of all pairs $(w,k)$ consisting of a word $w = (w_1, \ldots, w_n) \in \PP^n$ and an index $k \in [n]$ satisfying
		\begin{itemize}
			\item[(i)] 
			The word $w$ is ``cyclically weakly increasing'' from index $k$, namely
			$w_k \le w_{k+1} \le \ldots \le w_n \le w_1 \le \ldots \le w_{k-1}$.
			\item[(ii)] 
			If $j \in J \setminus \{k-1\}$ then $w_j < w_{j+1}$, where indices are computed modulo $n$. 
		\end{itemize}
	\end{definition}
	
	\begin{remark}
		The index $k$ is uniquely determined by the word $w$, unless all the letters of $w$ are equal; in which case, any index $k \in [n]$ will do by (i), but (ii) implies that either $J= \{k-1\}$ or $J = \varnothing$. 
		This holds even for $n = 1$, when either $J = \{0\} = \{1\}$ (computation modulo $1$) or $J = \varnothing$.
		Each of the ``words with repeated letters'' is therefore counted in $P^{\cyc}_{n,J}$ just once if $\# J = 1$, but $n$ times if $J = \varnothing$.
	\end{remark}
	
	\begin{example}
		Let $n=5$ and $J=\{1,3\}$. The pairs
		$(12345,1)$, $(23312,4)$ and $(23122,3)$ 
		are in $P^{\cyc}_{5,\{1,3\}}$ (see Figure~\ref{fig:1}), 
		but the pairs $(12354,1)$, $(22312,4)$ and $(23112,3)$ are not. 
		
		\begin{figure}[htb]
			\begin{center}
				\begin{tikzpicture}[scale=0.35]
				\draw[red] (0,0) circle (4);
				\draw[fill] (69:4) circle (.1); \draw[red](69:5) node
				{$1$}; \draw[fill] (34.5:4) circle (.2);
				\draw[fill] (0:4) circle (.1); \draw[blue] (0:5) node
				{$2$};
				\draw[fill] (-69:4) circle (.1); \draw[blue] (-69:5) node
				{$3$}; \draw[fill] (-106.5:4) circle (.2);
				\draw[fill] (-144:4) circle (.1); \draw[blue] (-144:5) node
				{$4$};
				\draw[fill] (144:4) circle (.1);\draw[blue] (144:5) node
				{$5$};
				\draw[red] (180:4) node {$\wedge$};
				\end{tikzpicture}
				\ \ \ \ \ \
				\begin{tikzpicture}[scale=0.35]
				\draw[red] (0,0) circle (4);
				\draw[fill] (69:4) circle (.1); \draw[blue](69:5) node
				{$2$}; \draw[fill] (34.5:4) circle (.2);
				\draw[fill] (0:4) circle (.1); \draw[blue] (0:5) node
				{$3$};
				\draw[fill] (-69:4) circle (.1); \draw[blue] (-69:5) node
				{$3$}; \draw[fill] (-106.5:4) circle (.2);
				\draw[fill] (-144:4) circle (.1); \draw[red] (-144:5) node
				{$1$};
				\draw[fill] (144:4) circle (.1);\draw[blue] (144:5) node
				{$2$};
				\draw[red] (180:4) node {$\wedge$};
				\end{tikzpicture}
				\ \ \ \ \ \
				\begin{tikzpicture}[scale=0.35]
				\draw[red] (0,0) circle (4);
				\draw[fill] (69:4) circle (.1); \draw[blue](69:5) node
				{$2$}; \draw[fill] (34.5:4) circle (.2);
				\draw[fill] (0:4) circle (.1); \draw[blue] (0:5) node
				{$3$};
				\draw[fill] (-69:4) circle (.1); \draw[red] (-69:5) node
				{$1$}; \draw[fill] (-106.5:4) circle (.2);
				\draw[fill] (-144:4) circle (.1); \draw[blue] (-144:5) node
				{$2$};
				\draw[fill] (144:4) circle (.1);\draw[blue] (144:5) node
				{$2$};
				\draw[red] (180:4) node {$\wedge$};
				\end{tikzpicture}
			\end{center}
			\caption{The pairs $(w,k)=(12345,1)$, $(23312,4)$ and $(23122,3)$ in $P^{\cyc}_{5,\{1,3\}}$.  In each case, $w_1$ is the value at the top right of the circle, while $w_k=1$ is shown red. The larger black dots indicate the positions $j$ in $J=\{1,3\}$ that require a strict ascent $w_j < w_{j+1}$ clockwise, unless $j=k-1$ as in the middle picture.}
			\label{fig:1}
		\end{figure}
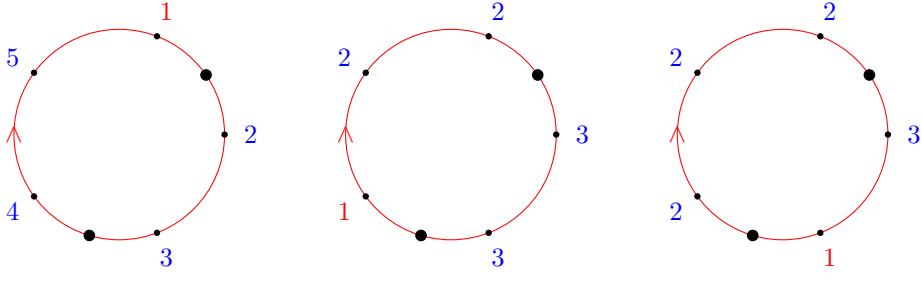		 
		
	\end{example}
	
	\begin{defn}\label{def:Fcyc} 	
		The {\em fundamental cyclic quasi-symmetric function} corresponding to a subset $J \subseteq [n]$ is defined by
		\[
		F^{\cyc}_{n,J}:=\sum_{(w,k) \in P^{\cyc}_{n,J}} x_{w_1} x_{w_2}\cdots x_{w_n}.
		\]
	\end{defn}
	
	\begin{example}\label{example:Fcyc}
		To illustrate Definition~\ref{def:Fcyc}, let us explain why
		\[
		\begin{aligned}
		F^{\cyc}_{5,\{1,3,5\}} 
		&= 5 \sum_{i_1<i_2<i_3<i_4<i_5} x_{i_1} x_{i_2} x_{i_3} x_{i_4} x_{i_5} \\
		&\quad + 2 \sum_{i_1<i_2<i_3<i_4} 
		(
		x_{i_1}^2 x_{i_2} x_{i_3} x_{i_4}
		+ x_{i_1} x_{i_2}^2 x_{i_3} x_{i_4}
		+ x_{i_1} x_{i_2} x_{i_3}^2 x_{i_4}
		+ x_{i_1} x_{i_2} x_{i_3} x_{i_4}^2
		) \\
		&\quad + \sum_{i_1<i_2<i_3} 
		(
		x_{i_1}^2 x_{i_2}^2 x_{i_3}
		+ x_{i_1} x_{i_2}^2 x_{i_3}^2
		+ x_{i_1}^2 x_{i_2} x_{i_3}^2
		).
		\end{aligned}
		\]
		For example, the monomial $x_1 x_2 x_3 x_4 x_5$ has coefficient $5$ above due to the elements
		\[
		((1,2,3,4,5),1), \,\,
		((5,1,2,3,4),2), \,\,
		((4,5,1,2,3),3),\,\,
		((3,4,5,1,2),4),\,\,
		((2,3,4,5,1),5)
		\]
		of $P^{\cyc}_{5,\{1,3,5\}}$,
		while the monomial $x_1^2 x_2 x_3 x_4$ has coefficient $2$ because of the elements
		\[
		((4,1,1,2,3),2), \,\,
		((2,3,4,1,1),4)
		\]
		of $P^{\cyc}_{5,\{1,3,5\}}$.
		The monomial $x_1^2 x_2^2 x_3$ has coefficient $1$ because of the element $((3,1,1,2,2),2)$ of $P^{\cyc}_{5,\{1,3,5\}}$.
		Readers may wish to check their understanding by verifying that
		\[
		\begin{aligned}
		F^{\cyc}_{6,\{2,4,6\}} 
		&= 6 \sum_{i_1<i_2<i_3<i_4<i_5<i_6} x_{i_1} x_{i_2} x_{i_3} x_{i_4} x_{i_5} x_{i_6} \\
		&\quad + 3 \sum_{i_1<i_2<i_3<i_4<i_5} 
		x_{i_1} x_{i_2} x_{i_3} x_{i_4} x_{i_5} 
		(x_{i_1} + x_{i_2} + x_{i_3} + x_{i_4} + x_{i_5}) \\
		&\quad + 3 \sum_{i_1<i_2<i_3<i_4} 
		(
		x_{i_1}^2 x_{i_2}^2 x_{i_3} x_{i_4}
		+ x_{i_1} x_{i_2}^2 x_{i_3}^2 x_{i_4}
		+ x_{i_1} x_{i_2} x_{i_3}^2 x_{i_4}^2
		+ x_{i_1}^2 x_{i_2} x_{i_3} x_{i_4}^2 
		) \\
		&\quad + 3 \sum_{i_1<i_2<i_3} 	x_{i_1}^2x_{i_2}^2 x_{i_3}^2. 
		\end{aligned}
		\]
	\end{example}

	
	By definition, $M^{\cyc}_{n,J}$ and
	$F^{\cyc}_{n,J}$ lie in $\cQSym_n$.
	The following lemma relates these two families.
	
	\begin{lemma}\label{t:Fcyc_vs_Mcyc}
		{\rm (Fundamental vs.\ monomial)}
		For any subset $J \subseteq [n]$
		\[
		F^{\cyc}_{n,J} 
		= \sum_{K \supseteq J} M^{\cyc}_{n,K} 
		\]
		and
		\[
		M^{\cyc}_{n,J} 
		= \sum_{K \supseteq J} (-1)^{\#(K \setminus J)} F^{\cyc}_{n,K}.
		\]
	\end{lemma}
	
	\begin{proof}
		For each subset $J \subseteq [n]$ denote by $Q^{\cyc}_{n,J}$ the set of all pairs $(w,k)$ consisting of a word $w = (w_1, \ldots, w_n) \in \PP^n$ and an index $k \in [n]$ which satisfy
		\begin{itemize}
			\item[(i)] 
			The word $w$ is ``cyclically weakly increasing'' from index $k$, namely
			$w_k \le w_{k+1} \le \dots \le w_n \le w_1 \le \dots \le w_{k-1}$.
			\item[(ii)] 
			$k-1 \in J$ and, for $j \in [n] \setminus \{k-1\}$: $j \in J \setminus \{k-1\} \iff w_j < w_{j+1}$, where indices are computed modulo $n$. 
			(Thus $Q^{\cyc}_{n,J}= \varnothing$ for $J = \varnothing$.)
		\end{itemize}
		Definition~\ref{def:Mcyc} can now be written in the form
		\[
		M^{\cyc}_{n,J} =
		\sum_{(w,k) \in Q^{\cyc}_{n,J}} x_{w_1} x_{w_2}\cdots x_{w_n}
		\qquad (\forall J \subseteq [n]).
		\]
		The sets $P^{\cyc}_{n,J}$ from Definition~\ref{def:cyclic_word} are clearly disjoint unions
		\[
		P^{\cyc}_{n,J} = \bigsqcup_{K \supseteq J} Q^{\cyc}_{n,K}
		\qquad (\forall J \subseteq [n])
		\]
		and thus, by Definition~\ref{def:Fcyc},
		\[
		F^{\cyc}_{n,J} = \sum_{K \supseteq J} M^{\cyc}_{n,K} 
		\qquad (\forall J \subseteq [n]).
		\]
		The other claim follows by the principle of inclusion-exclusion.
	\end{proof}
	
	\begin{proposition}\label{t:Fcyc_from_F}
		{\rm (From fundamental to cyclic fundamental)}
		For any subset $J \subseteq [n]$
		\[
		F^{\cyc}_{n,J} 
		= \sum_{i \in [n]} F_{n, (J - i) \cap [n-1]},
		\]
		where $J - i := \{k - i \,:\, k \in J\} \subseteq [n]$ with subtraction interpreted cyclically modulo $n$.
	\end{proposition}
	
	\begin{proof}
		For any $J \subseteq [n]$, by Lemma~\ref{t:Mcyc_from_M} and Lemma~\ref{t:Fcyc_vs_Mcyc},
		\[
		F^{\cyc}_{n,J} 
		= \sum_{J \subseteq K \subseteq [n]} M^{\cyc}_{n,K}
		= \sum_{J \subseteq K \subseteq [n]} \,\, \sum_{i \in K} 
		M_{n,(K-i) \cap [n-1]}
		= \sum_{i \in [n]}  \,\, \sum_{J \cup \{i\} \subseteq K \subseteq [n]}
		M_{n,(K-i) \cap [n-1]}.
		\]
		Denoting $K' := K-i$ and $K'' := K' \cap [n-1]$, it follows that
		\begin{eqnarray*}
			F^{\cyc}_{n,J} 
			&=& \sum_{i \in [n]}  \,\, 
			\sum_{(J-i) \cup \{n\} \subseteq K' \subseteq [n]} 
			M_{n, K' \cap [n-1]} \\
			&=& \sum_{i \in [n]}  \,\, 
			\sum_{(J-i) \cap [n-1] \subseteq K'' \subseteq [n-1]} 
			M_{n, K''}
			= \sum_{i \in [n]} F_{n, (J - i) \cap [n-1]}.
		\end{eqnarray*}
	\end{proof}
	
	\begin{example}
		For the functions in Example~\ref{example:Fcyc},
		\[
		F^{\cyc}_{5,\{1,3,5\}} 
		= F_{5,\{1,3\}} + F_{5,\{2,4\}} +  F_{5,\{1,3,4\}} + F_{5,\{2,3\}} + F_{5,\{1,2,4\}}
		\]
		and
		\[
		F^{\cyc}_{6,\{2,4,6\}} 
		= 3 F_{6,\{2,4\}} + 3 F_{6,\{1,3,5\}}.
		\]
	\end{example}
	
	\begin{lemma}\label{t:Fcyc_properties}\ \\
		\begin{enumerate}
			\item
			$F^{\cyc}_{n,\varnothing} = n F_{n,\varnothing} = n h_n$
			and $F^{\cyc}_{n,[n]} = n F_{n,[n-1]} = n e_n$
			are symmetric functions.
			\item
			Cyclic invariance:
			If $J' \in 2^{[n]}$ is a cyclic shift of $J \in 2^{[n]}$ then 
			$F^{\cyc}_{n,J'} = F^{\cyc}_{n,J}$.
			\item
			Linear dependence:
			$\sum_{J \subseteq [n]} (-1)^{\#J} F^{\cyc}_{n,J} = 0$.
		\end{enumerate}
	\end{lemma}
	
	\begin{proof}
		(1) and (2) are immediate from Proposition~\ref{t:Fcyc_from_F}.
		(3) follows from the second formula in Lemma~\ref{t:Fcyc_vs_Mcyc}, for $J = \varnothing$, since $M^{\cyc}_{n,\varnothing} = 0$.
	\end{proof}

	\subsection{Normalization}\label{sec:normalization}
	
	In order to get a basis for the vector space $\cQSym_n \otimes \,\QQ$, we need to take a suitable set of representatives of the functions defined in the previous subsections.
	In order to get an actual $\ZZ$-basis for $\cQSym_n$, we also need to normalize them.  
	We first recall and augment some of our previously used notations, as follows:  
	$2^{[n]}$ denotes the set of all subsets of $[n]=\{1,2,\ldots,n\}$, 
	and $2^{[n]}_0$ denotes the set of nonempty subsets of $[n]$, while $c2^{[n]}$ and $c2^{[n]}_0$,
	respectively, denote the corresponding collections of equivalence classes under cyclic shifts.
	
	\begin{defn}\label{def:Mcyc_normalized}
		For any $J \subseteq [n]$ let 
		$D_J := \{i \in \ZZ/n\ZZ \,:\, J+i \equiv J \pmod{n} \}$
		be the {\em stabilizer} of $J$ under the action of $\ZZ/n\ZZ$ by cyclic shifts, and let $d_J := \#D_J$.
		Define the {\em normalized monomial cyclic quasi-symmetric function} $\hM^{\cyc}_{n,J}$ by
		\[
		\hM^{\cyc}_{n,J} := \frac{1}{d_J} M^{\cyc}_{n,J}
		\qquad (\forall J \subseteq [n]).
		\]
		For any orbit $A \in c2^{[n]}$ define
		\[
		M^{\cyc}_{n,A} := M^{\cyc}_{n,J}, \quad
		D_A := D_J, \quad
		d_A := d_J, \quad \text{and} \quad
		\hM^{\cyc}_{n,A} := \hM^{\cyc}_{n,J},
		\]
		where $J$ is any element of the orbit $A$. By Observation~\ref{t:Mcyc_cyclic_invariance}, these are all well-defined (i.e., independent of the choice of $J$ in $A$).
		Note also that, for each $A \in c2^{[n]}$, $d_A \cdot \#A = n$, and therefore
		\[
		\hM^{\cyc}_{n,A} = \frac{1}{n} \sum_{J \in A} M^{\cyc}_{n,J}
		\qquad (\forall A \in c2^{[n]}).
		\]
	\end{defn}
	
	\begin{example}
		For $n = 6$ and $A = \{\{1,4\}, \{2,5\},\{3,6\}\} \in c2^{[6]}$, $\cc(A) = (3,3)$ and
		\[
		M^{\cyc}_{n,A} = 2 \sum_{i_1 < i_2} x_{i_1}^3 x_{i_2}^3.
		\]
		Since $d_A = 2$,
		\[
		\hM^{\cyc}_{n,A} = \sum_{i_1 < i_2} x_{i_1}^3 x_{i_2}^3.
		\]
	\end{example}
	
	\begin{lemma}\label{t:Mcyc_basis}
		$\left\{ \hM^{\cyc}_{n,A} \,:\, A \in \cnes \right\}$ is a $\ZZ$-basis for $\cQSym_n$. 
	\end{lemma}
	\begin{proof}
		It follows from Definition~\ref{def:Mcyc} that, for any $A \in \cnes$, the coefficient of each monomial in $M^{\cyc}_{n,A}$ is $d_A$.
		Therefore $\hM^{\cyc}_{n,A} \in \cQSym_n$ is a sum of monomials with all coefficients equal to $1$. 
		The set $\left\{ \hM^{\cyc}_{n,A} \,:\, A \in \cnes \right\}$ clearly $\ZZ$-spans $\cQSym_n$ and is linearly independent,
		since each monomial of degree $n$ appears in $\hM^{\cyc}_{n,A}$ for exactly one $A \in \cnes$.
		Note also that $\hM^{\cyc}_{n,\{\varnothing\}} = M^{\cyc}_{n,\varnothing} = 0$.
	\end{proof}
	
	\begin{defn}\label{def:Fcyc_normalized}
		For $J \subseteq [n]$ let $D_J$ and $d_J$ be as in Definition~\ref{def:Mcyc_normalized}.
		Define the {\em normalized fundamental cyclic quasi-symmetric function} $\hF^{\cyc}_{n,J}$ by
		\[
		\hF^{\cyc}_{n,J} := \frac{1}{d_J} F^{\cyc}_{n,J}
		\qquad (\forall J \subseteq [n]).
		\]
		For any orbit $A \in c2^{[n]}$ define
		\[
		F^{\cyc}_{n,A} := F^{\cyc}_{n,J} \quad \text{and} \quad
		\hF^{\cyc}_{n,A} := \hF^{\cyc}_{n,J},
		\]
		where $J$ is any element of the orbit $A$. By Lemma~\ref{t:Fcyc_properties}(2), these are all well-defined (i.e., independent of the choice of $J \in A$).
		Also, as for the monomial functions,
		\[
		\hF^{\cyc}_{n,A} = \frac{1}{n} \sum_{J \in A} F^{\cyc}_{n,J}
		\qquad (\forall A \in c2^{[n]}).
		\]
	\end{defn}
	
	\begin{lemma}\label{t:Fcyc_basis}\ \\
		\begin{enumerate}
			\item
			$\hF^{\cyc}_{n,\varnothing} = F_{n,\varnothing} = h_n$
			and $\hF^{\cyc}_{n,[n]} = F_{n,[n-1]} = e_n$
			are symmetric functions.
			\item
			Linear dependence:
			\[
			\sum_{A \in c2^{[n]}} (-1)^{r(A)} \hF^{\cyc}_{n,A} = 0,
			\]
			where the {\em rank} of an orbit $A \in c2^{[n]}$ is $r(A) := \#J$, for any element $J \in A$; the notation is meant to distinguish $r(A)$ from the size $\#A$ of the orbit $A$.
			\item
			For any $C \in c2^{[n]}$, the set 
			$\BBB_C := \left\{ \hF^{\cyc}_{n,A} \,:\, A \in c2^{[n]} \setminus \{C\} \right\}$ is a $\ZZ$-basis for $\cQSym_n$. In particular, for $C = [\varnothing]$,
			$\BBB_0 := \left\{ \hF^{\cyc}_{n,A} \,:\, A \in \cnes \right\}$ is a $\ZZ$-basis for $\cQSym_n$. 
		\end{enumerate}
	\end{lemma}
	
	\begin{proof}
		(1) follows immediately from Lemma~\ref{t:Fcyc_properties}(1) and Definition~\ref{def:Fcyc_normalized}, since $d_{\varnothing} = d_{[n]} = n$.
		
		%
		
		(2) Write Lemma~\ref{t:Fcyc_properties}(3) in the form
		\[
		\sum_{J \subseteq [n]} (-1)^{\#J} d_J \hF^{\cyc}_{n,J} = 0,
		\]
		or equivalently
		\[
		\sum_{A \in c2^{[n]}} (-1)^{r(A)} \#A \cdot d_A \hF^{\cyc}_{n,A} = 0.
		\]
		Clearly
		\[
		\#A \cdot d_A = n
		\]
		for any orbit $A \in c2^{[n]}$, yielding the claimed formula.
		
		(3) Note that, at this point, we know that $\hF^{\cyc}_{n,A} \in \cQSym_n\! \otimes \,\QQ$ $(\forall A \in c2^{[n]})$ 
		but not necessarily $\hF^{\cyc}_{n,A} \in \cQSym_n$. This is part of the claim, and will be proved below.
		
		For $A, B \in c2^{[n]}$ write $A \le B$ if 
		there exist $J \in A$ and $K \in B$ such that $J \subseteq K$
		(equivalently,
		for any $J \in A$ there exists $K \in B$ such that $J \subseteq K$;
		equivalently,
		for any $K \in B$ there exists $J \in A$ such that $J \subseteq K$).
		It is easy to see that $\le$ is a partial order on $c2^{[n]}$.
		For $J, K \subseteq [n]$, let $d_{J,K}:=\#D_{J,K}$ where:
		\begin{equation}
			\label{set-defining-DJK}
			D_{J,K}:=
			\{i \in \ZZ/n\ZZ \,:\, J \subseteq K+i\}
			=\{i \in \ZZ/n\ZZ \,:\, J-i \subseteq K\}.
		\end{equation}
		The number $d_{J,K}$ is invariant under cyclic shifts of either $J$ or $K$, and can therefore be used to define $d_{A,B}$ for any $A, B \in c2^{[n]}$. Clearly, $d_{A,B} > 0 \iff A \le B$; note that $d_{A,A} = d_A$.  We claim that $d_{A,B}$ is divisible by both $d_A$ and $d_B$: 
		if one fixes representatives $J$ and $K$ for $A$ and $B$, respectively, then one has
		\begin{itemize}
			\item[(i)] $\frac{d_{A,B}}{d_B}=\#\{K' \in B: J \subseteq K'\}$, as
			$i  \mapsto  K+i$ gives a $d_B$-to-$1$ surjection $D_{J,K}
			\twoheadrightarrow \{K' \in B: J \subseteq K'\}$,
			\item[(ii)] $\frac{d_{A,B}}{d_A}=\#\{J' \in A: J' \subseteq K\}$, as
			$i  \mapsto  J-i$ gives a $d_A$-to-$1$ surjection $D_{J,K}
			\twoheadrightarrow \{J' \in A: J' \subseteq K\}$.
		\end{itemize}
		Then the first formula in Lemma~\ref{t:Fcyc_vs_Mcyc} provides the second equality in the following rewriting:
		\begin{equation}
			\label{long-rewriting}
			F^{\cyc}_{n,A}=F^{\cyc}_{n,J}
			= \sum_{\substack{K' \subseteq [n]: \\ J \subseteq K'}} M^{\cyc}_{n,K'}
			= \sum_{B \ge A} M^{\cyc}_{n,B} \cdot \#\{K' \in B:\\J \subseteq K'\} 
			= \sum_{B \ge A} \frac{d_{A,B}}{d_B} M^{\cyc}_{n,B} 
		\end{equation}
		where the last equality uses (i) above.  
		Dividing the far left and right sides of \eqref{long-rewriting} by $d_A$ gives: 
		\begin{equation}
			\label{eq:111}
			\hF^{\cyc}_{n,A} 
			= \sum_{B \ge A} \frac{d_{A,B}}{d_A} \hM^{\cyc}_{n,B}.
		\end{equation}
		An analogous rewriting using the second formula in Lemma~\ref{t:Fcyc_vs_Mcyc} leads to the following:
		\begin{equation}\label{eq:112}
			\hM^{\cyc}_{n,A} 
			= \sum_{B \ge A} (-1)^{r(B) - r(A)}\frac{d_{A,B}}{d_A} \hF^{\cyc}_{n,B}.
		\end{equation}
		
		It now follows from Equation~\eqref{eq:111} together with Lemma~\ref{t:Mcyc_basis} that $\hF^{\cyc}_{n,A} \in \cQSym_n$ for each $A \in c2^{[n]}$.
		Furthermore, Equation~\eqref{eq:112} and Lemma~\ref{t:Mcyc_basis} imply
		that $\left\{ \hF^{\cyc}_{n,A} \,:\, A \in c2^{[n]} \right\}$ spans $\cQSym_n$ over $\ZZ$. 
		Since $\cQSym_n$ is a torsion-free abelian group of rank $\#\cnes$,
		by Lemma~\ref{t:Mcyc_basis}, the linear dependence (2) above completes the proof.
	\end{proof}
	
	
	\begin{example}
		The matrix whose columns expand $\{F^{\cyc}_{n,J}\}$
		in terms of $\{M^{\cyc}_{n,K}\}$ is given below for $n=2$, $3$ and $4$, indexing rows and columns by representatives of the equivalence classes in $\cnes$. 
		\[
		\bordermatrix{~ 
			~ & F^{\cyc}_{2,\{1,2\}} 
			&  F^{\cyc}_{2,\{1\}} \cr 
			M^{\cyc}_{2,\{1,2\}} &1 &1 \cr
			M^{\cyc}_{2,\{1\}} &0 &1
		}
		\]
		\vskip.2in
		\[
		\bordermatrix{~ 
			~ & F^{\cyc}_{3,\{1,2,3\}} 
			&  F^{\cyc}_{3,\{1,2\}}
			&  F^{\cyc}_{3,\{1\}} \cr 
			M^{\cyc}_{3,\{1,2,3\}} &1 &1 &1 \cr
			M^{\cyc}_{3,\{1,2\}} &0 &1 &2 \cr
			M^{\cyc}_{3,\{1\}} &0 &0 &1  \cr
		}
		\]
		\vskip.2in
		\[
		\bordermatrix{~ 
			~ & F^{\cyc}_{4,\{1,2,3,4\}} 
			&  F^{\cyc}_{4,\{1,2,3\}}
			&  F^{\cyc}_{4,\{1,3\}}
			&  F^{\cyc}_{4,\{1,2\}}
			& F^{\cyc}_{4,\{1\}} \cr 
			M^{\cyc}_{4,\{1,2,3,4\}} &1 &1 &1 &1 &1 \cr
			M^{\cyc}_{4,\{1,2,3\}} &0 &1 &2 &2 &3 \cr
			M^{\cyc}_{4,\{1,3\}} &0 &0 &1 &0 &1 \cr
			M^{\cyc}_{4,\{1,2\}} &0 &0 &0 &1 &2 \cr
			M^{\cyc}_{4,\{1\}} &0 &0 &0 &0 &1 \cr
		}
		\]
		\vskip.2in
		Here are the same matrices for 
		$\{\hF^{\cyc}_{n,J}\}$ and $\{\hM^{\cyc}_{n,K}\}$:
		\[
		\bordermatrix{~ 
			~ & \hF^{\cyc}_{2,\{1,2\}}  
			&  \hF^{\cyc}_{2,\{1\}} \cr 
			\hM^{\cyc}_{2,\{1,2\}} &1 &2 \cr
			\hM^{\cyc}_{2,\{1\}} &0 &1
		}
		\]
		\vskip.2in
		\[
		\bordermatrix{~ 
			~ & \hF^{\cyc}_{3,\{1,2,3\}}
			&  \hF^{\cyc}_{3,\{1,2\}}
			&  \hF^{\cyc}_{3,\{1\}} \cr 
			\hM^{\cyc}_{3,\{1,2,3\}} &1 &3 &3 \cr
			\hM^{\cyc}_{3,\{1,2\}} &0 &1 &2 \cr
			\hM^{\cyc}_{3,\{1\}} &0 &0 &1  \cr
		}
		\]
		\vskip.2in
		\[
		\bordermatrix{~ 
			~ & \hF^{\cyc}_{4,\{1,2,3,4\}}
			&  \hF^{\cyc}_{4,\{1,2,3\}}
			&  \hF^{\cyc}_{4,\{1,3\}} 
			&  \hF^{\cyc}_{4,\{1,2\}}
			& \hF^{\cyc}_{4,\{1\}} \cr 
			\hM^{\cyc}_{4,\{1,2,3,4\}}  &1 &4 &2 &4 &4 \cr
			\hM^{\cyc}_{4,\{1,2,3\}} &0 &1 &1 &2 &3 \cr
			\hM^{\cyc}_{4,\{1,3\}}  &0 &0 &1 &0 &2 \cr
			\hM^{\cyc}_{4,\{1,2\}} &0 &0 &0 &1 &2 \cr
			\hM^{\cyc}_{4,\{1\}} &0 &0 &0 &0 &1 \cr
		}
		\]
	\end{example}
	
	\begin{remark}\label{rem:cQSym_0}
		The degree $0$ homogeneous component $\cQSym_0$ is, of course, isomorphic to $\ZZ$.
		We formally define $F^{\cyc}_{0,\varnothing} = M^{\cyc}_{0,\varnothing} = 1$ to get monomial and fundamental bases for it.
	\end{remark}

	\subsection{The non-Escher subring $\cQSym^-$}
	
	For many combinatorial applications, in which the non-Escher property of cyclic descent sets (see Definition~\ref{def:cDes}) is manifest, 
	it is natural to consider a certain subgroup $\cQSym_n^-$ of $\cQSym_n$. 
	
	\begin{defn}\label{def:cQSym-}
		For any positive integer $n>1$, let $2_{0,n}^{[n]}$ be the set of all subsets of $[n]$ other than the empty set and $[n]$ itself.
		Let $c2_{0,n}^{[n]}$ be the set of equivalence classes of elements of $2_{0,n}^{[n]}$ under cyclic shifts.
		Define, for each $n > 1$,
		\[
		\cQSym_n^- 
		= \spn_{\ZZ} \left\{ \hF^{\cyc}_{n,J} \,:\, J \in 2_{0,n}^{[n]} \right\}.
		\]
	\end{defn}
	
	By Lemma~\ref{t:Fcyc_basis}, $\cQSym_n^-$ is a free abelian subgroup of corank one in $\cQSym_n$.
	Define also
	$\cQSym_1^- 
	:= \spn_{\ZZ} \left\{\hF^{\cyc}_{1,\varnothing}\right\}$ where  $\hF^{\cyc}_{1,\varnothing} = e_1 = h_1$, 
	$\cQSym_0^- := \ZZ$, and
	\[
	\cQSym^- := \bigoplus_{n \ge 0} \cQSym_n^-.
	\]
	
	
	\begin{obs}\label{t:proper_basis}
		For each $n > 1$, 
		the set $\left\{ \hF^{\cyc}_{n,A} \,:\, A \in c2_{0,n}^{[n]} \right\}$ is a $\ZZ$-basis for $\cQSym_n^-$.
	\end{obs}
	
	In fact, $\cQSym$ is a ring and $\cQSym^-$ a subring; see Proposition~\ref{t:cQSym_is_a_ring}.
	We call $\cQSym^-$ the {\em non-Escher subring} of $\cQSym$.

	\section{Toric posets and cyclic $P$-partitions}
	\label{sec:toric_posets}
	
	In this section we take a little excursion into cyclic analogues of posets that were called {\em toric posets} \cite{DMR}, and develop a theory of cyclic
	$P$-partitions for them; in particular, 
	a cyclic analogue of Stanley's fundamental decomposition lemma for $P$-partitions  
	is provided.
	Just as the fundamental quasi-symmetric functions $F_{n,J}$ are $P$-partition enumerators for certain (labeled) total orders, 
	the fundamental cyclic quasi-symmetric functions
	$F^{\cyc}_{n,J}$ will be shown to be cyclic $P$-partition enumerators for certain (labeled) total cyclic orders. 
	This will be used to prove that $\cQSym^-$ is a ring and to study its structure constants.

	\subsection{Ordinary $P$-partitions for directed acyclic graphs}
	
	We first review here $P$-partitions for a labeled poset, rephrased very slightly in terms of a {\em directed acyclic graph}.  Although this rephrasing is trivial, it is better for thinking later about the cyclic/toric version.
	
	\begin{definition}
		A {\it directed acyclic graph} (DAG) on $\{1,2,\ldots,n\}$ is a subset $\DAG{D}$ of the cartesian product $\{1,2,\ldots,n\} \times \{1,2,\ldots,n\}$, or a binary relation (written either as $i \rightarrow j$ in $\DAG{D}$, or $i \overset{\DAG{D}}{\rightarrow} j$) containing no directed cycles $i_1 \rightarrow i_2 \rightarrow \cdots \rightarrow i_{k-1} \rightarrow i_k=i_1$ with $k >1$.  In particular, it is
		\begin{itemize}
			\item {\it antisymmetric}, that is, one cannot have
			both $i \rightarrow j$ and $j \rightarrow i$ in $\DAG{D}$, and
			\item {\it irreflexive}, that is, one cannot have $i \rightarrow i$ in $\DAG{D}$.
		\end{itemize}
		
		When one has two DAGs $\DAG{D}_1, \DAG{D}_2$ (on the same set of vertices) with $\DAG{D}_1 \subseteq \DAG{D}_2$, then one says that $\DAG{D}_2$ {\it extends} $\DAG{D}_1$.
		
		A DAG $\DAG{P}$ is {\it transitive} if $i \rightarrow j$ and  $j \rightarrow k$ in $\DAG{P}$ implies $i \rightarrow k$ in $\DAG{P}$.
		Transitive DAGs are called {\it posets}.
		
		The {\it transitive closure} $\DAG{P}$ of a DAG $\DAG{D}$ is the poset extending $\DAG{D}$ obtained from $\DAG{D}$ by adding in $i_1 \rightarrow i_k$
		whenever one has a chain $i_1 \rightarrow  i_2 \rightarrow \cdots \rightarrow i_{k-1} \rightarrow i_k$ in $\DAG{D}$.
		
		A poset $\DAG{P}$ is called a {\em total} (or {\em linear}) {\em order} if for every pair $(i,j)$ with $1 \leq i < j \leq n$, 
		it contains either $i \rightarrow j$ or $j \rightarrow i$.  
		In other words, $\DAG{P}$ is a {\it transitive tournament}.  
		It is easily seen that this means that $\DAG{P}=\DAG{w}$ is the transitive closure of the DAG having $w_1 \rightarrow w_2 \rightarrow \cdots \rightarrow w_n$ for some (unique) permutation $w=(w_1,w_2,\ldots,w_n)$ in the symmetric group $\symm_n$.
		
		Denote by $\LLL(\DAG{D})$ the set of all permutations $w$ in $\symm_n$ for which $\DAG{w}$ extends $\DAG{D}$.
	\end{definition}
	
	\begin{definition}
		A {\it $\DAG{D}$-partition} is a function $f: \{1,2,\ldots,n\} \rightarrow \{1,2,\ldots\}$ for which
		\begin{itemize}
			\item $f(i) \leq f(j)$ whenever $i \rightarrow j$ in $\DAG{D}$, and
			\item $f(i) < f(j)$ whenever $i \rightarrow j$ in $\DAG{D}$, but $i >_{\ZZ} j$.
		\end{itemize}
		
		Denote by $\AAA(\DAG{D})$ the set of all $\DAG{D}$-partitions $f$.
	\end{definition}

	\begin{lemma}(Fundamental lemma of $\DAG{D}$-partitions \cite[Lemma 3.15.3]{EC1})
		\label{fundamental-lemma-of-P-partitions}
		For any DAG $\DAG{D}$, one has a decomposition of $\AAA(\DAG{D})$ as the following disjoint union:
		\[
		\AAA(\DAG{D}) = \bigsqcup_{ w \in \LLL(\DAG{D}) } \AAA(\DAG{w}).
		\]
	\end{lemma}
	
	\noindent
	The following proposition explains why the theory of DAGs is the same as that of posets.
	\begin{proposition}
		If $\DAG{D}_2$ extends $\DAG{D}_1$ then one has inclusions
		\[
		\begin{aligned}
		\LLL(\DAG{D}_2) &\subseteq \LLL(\DAG{D}_1),\\
		\AAA(\DAG{D}_2) & \subseteq \AAA(\DAG{D}_1),
		\end{aligned}
		\]
		with equality in both cases
		if $\DAG{D}_2$ is the transitive closure of $\DAG{D}_1$. 
	\end{proposition}

	\subsection{Toric DAGs, toric posets, and  toric $P$-partitions}\label{sec:toric}
	
	The toric case is mildly trickier, because one must consider the following equivalence relation on DAGs.
	
	\begin{definition}
		Say that a DAG $\DAG{D}$ has $i_0$ in $\{1,2,\ldots,n\}$ as a {\it source} (respectively, {\it sink}) if $\DAG{D}$ contains
		no arrows/relations of the form $j \rightarrow i_0$ (respectively, of the form $i_0 \rightarrow j$). 
		
		Say that $\DAG{D'}$
		is obtained from $\DAG{D}$ by an {\it elementary equivalence} or {\it flip at $i_0$} if $i_0$ is either a source or sink of $\DAG{D}$ and one obtains $\DAG{D}'$ by reversing all the arrows in $\DAG{D}$ incident with $i_0$. 
		Define the equivalence relation $\equiv$ on DAGs to be
		the reflexive-transitive closure of the elementary equivalences, that is, $\DAG{D} \equiv \DAG{D'}$ if and only if there exists a (possibly empty) sequence of flips one can apply starting with $\DAG{D}$ to obtain $\DAG{D'}$.
	\end{definition}
	
	\begin{definition}
		A {\it toric DAG} is the $\equiv$-equivalence class $[\DAG{D}]$ of a DAG $\DAG{D}$.
	\end{definition}
	
	\begin{example}
		Here is an example of a toric DAG $[\DAG{D}_1]$:
		
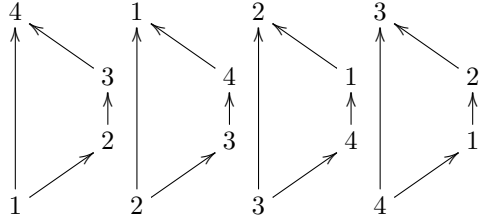
\[
		\xymatrix@R=12pt{
			4                & \\
			& 3 \ar[ul] \\
			& 2 \ar[u] \\
			1\ar[ur] \ar[uuu]
		} 
		\xymatrix@R=12pt{
			1                & \\
			& 4 \ar[ul] \\
			& 3 \ar[u] \\
			2\ar[ur] \ar[uuu]
		}
		\xymatrix@R=12pt{
			2                & \\
			& 1 \ar[ul] \\
			& 4 \ar[u] \\
			3\ar[ur] \ar[uuu]
		} 
		\xymatrix@R=12pt{
			3                & \\
			& 2 \ar[ul] \\
			& 1 \ar[u] \\
			4\ar[ur] \ar[uuu]
		}
		\]
		
		Here is another toric DAG $[\DAG{D}_2]$:
		$$
		
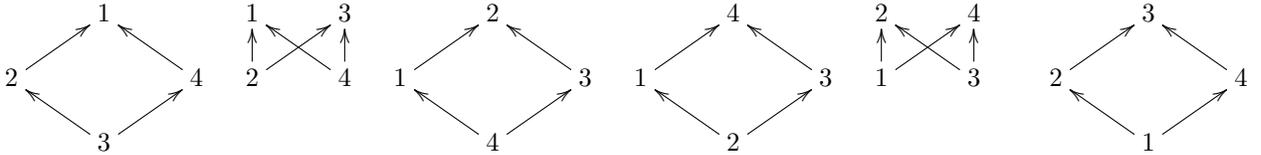
\xymatrix@R=12pt{
			& 1              & \\
			2\ar[ur]&                &4\ar[ul] \\
			& 3\ar[ur]\ar[ul]&
		} \quad
		\xymatrix@R=12pt{
			1              & 3 \\
			2\ar[u]\ar[ur] & 4\ar[u]\ar[ul]
		}\quad
		\xymatrix@R=12pt{
			& 2              & \\
			1\ar[ur]&                &3\ar[ul] \\
			& 4\ar[ur]\ar[ul]&
		}\quad
		\xymatrix@R=12pt{
			& 4              & \\
			1\ar[ur]&                &3\ar[ul] \\
			& 2\ar[ur]\ar[ul]&
		}
		\quad
		\xymatrix@R=12pt{
			2              & 4 \\
			1\ar[u]\ar[ur] & 3\ar[u]\ar[ul]
		}\quad
		\quad
		\xymatrix@R=12pt{
			& 3              & \\
			2\ar[ur]&                &4\ar[ul] \\
			& 1\ar[ur]\ar[ul]&
		}
		$$
	\end{example}

	\begin{definition}
		Say that $[\DAG{D}_2]$ {\it torically extends} $[\DAG{D}_1]$ if there exist $\DAG{D}_i' \in [\DAG{D}_i]$ for $i=1,2$ with 
		$\DAG{D'_1} \subseteq \DAG{D'_2}$.
	\end{definition}
	
	A certain toric extension, called the toric transitive closure, will be particularly important.
	\begin{definition}
		Say that $i \rightarrow j$ is implied from {\it toric transitivity} in a DAG $\DAG{D}$ if there exist in $\DAG{D}$ 
		both a chain $i_1 \rightarrow  i_2 \rightarrow \cdots \rightarrow i_k$ and a direct arrow $i_1 \rightarrow i_k$ 
		such that $i=i_a, j=i_b$ for some $1 \leq a < b \leq k$.
		
		The {\em toric transitive closure} of $\DAG{D}$ is the DAG $\DAG{P}$ obtained by adding in all arrows $i \rightarrow j$ implied from toric transitivity in $\DAG{D}$;
		see Figure~\ref{fig:toric-transitive-closure}, where the existence of all the solid arcs in $\DAG{D}$ implies the existence of all the dotted arcs in its toric transitive closure.
		Note that the toric transitive closure of $\DAG{D}$
		is a subset of its usual transitive closure. 
		
		A DAG $\DAG{D}$ is {\em toric transitively closed} if it equals its toric transitive closure\footnote{It is not hard to see, e.g., from \cite[Prop. 4.2]{DMR}, that the toric transitive closure of any DAG will already be toric transitively closed.  That is, one
			does not need to iterate the closure procedure to reach a DAG which is toric transitively closed.}.
		
	\end{definition}
	
	\begin{figure}
		\[
		\xymatrix@R=20pt{                                                                                                     
			i_k & & & &\\                                                                                                       
			& &i_{k-1} \ar[ull]& &\\                                                                        
			& & &i_{k-2} \ar[ul]\ar@{-->}[uulll] &\\                                                           
			& & & &\vdots \ar[ul]\ar@{-->}[uuullll] \\                                                                          
			& & &i_3 \ar[ur]\ar@{-->}[uuuulll] &\\                                                             
			& &i_2 \ar[ur] \ar@{-->}[uuuuull] \ar@{-->}[uuur]\ar@{-->}[uuuu]&  & \\                                                           
			i_1 \ar[uuuuuu] \ar[urr] \ar@{-->}[uurrr]\ar@{-->}[uuuurrr]\ar@{-->}[uuuuurr]\ar@{-->}[uuurrrr]&  &  &  &    }    
		\]
		\caption{Toric transitive closure}
		\label{fig:toric-transitive-closure}
	\end{figure}
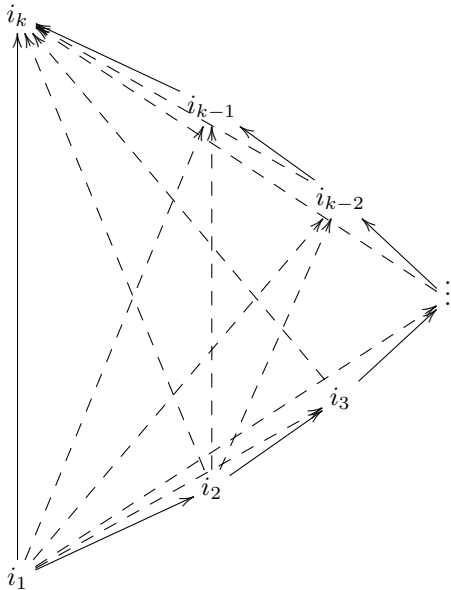
	
	\begin{proposition}
		\label{toric-closure-independent-of-representative}
		If $\DAG{D_1} \equiv \DAG{D_2}$, then $\DAG{D}_1$ is toric transitively closed if and only if the
		same is true for $\DAG{D}_2$.
	\end{proposition}
	\begin{proof}
		It suffices to check this when $\DAG{D}_1,\DAG{D}_2$ differ by a flip at some node $j$.  
		Note that for every set of vertices $\{i_1,\ldots,i_k\}$ 
		for which $\DAG{D}_1$ contains all the solid arcs shown in Figure~\ref{fig:toric-transitive-closure}, it will contain all of the dotted
		arcs before the flip at node $j$ if and only if $\DAG{D}_2$ contains the corresponding arcs after the flip; 
		this is obvious if $j \not\in \{i_1,\ldots,i_k\}$,
		and straightforward to check when $j \in \{i_1,\ldots,i_k\}$.
	\end{proof}
	
	\begin{definition}
		A toric DAG $[\DAG{D}]$ is a {\em toric poset}\footnote{This isn't quite the way that it was defined in \cite{DMR}, but essentially equivalent, via \cite[Theorem 1.4]{DMR}.} if $\DAG{D}$ is toric transitively closed for one of its 
		$\equiv$-class representatives  $\DAG{D}$, or equivalently, by Proposition~\ref{toric-closure-independent-of-representative}, for {\it all} such representatives  $\DAG{D}$. 
	\end{definition}
	
	\begin{definition}
		A {\em total cyclic order} is a toric poset $[\DAG{w}]$ for some $w=(w_1,\ldots,w_n)$ in $\symm_n$, that is,
		\[
		\begin{aligned}[]
		[\DAG{w}] =
		&\{ (w_1,w_2,\ldots,w_{n-1},w_n)^\rightharpoonup, \\
		&\,\,\, (w_2,w_3,\ldots,w_{n},w_1)^\rightharpoonup, \\
		&\,\,\, \qquad \qquad \vdots \\
		&\,\,\, (w_{n},w_1,\ldots,w_{n-2},w_{n-1})^\rightharpoonup\}.
		\end{aligned}
		\]
		In other words, a total cyclic order is a toric poset with
		at least one (equivalently,  all) of its $\equiv$-class
		representatives being a total (linear) order.
		
		Denote by $\LLL^\tor([\DAG{D}])$ the set of all total cyclic orders $[\DAG{w}]$ which torically extend $[\DAG{D}]$.
	\end{definition}
	
	\begin{remark}\label{rem:cyclic_perm}
		\begin{itemize}
			\item[1.]
			A total cyclic order $[\DAG{w}]$ may be viewed as a coset $w\ZZ_n \in \symm_n/\ZZ_n$, where $\ZZ_n$
			is the subgroup of $\symm_n$ generated by the $n$-cycle $(2,3,\dots,n,1)$.
			\item[2.]
			Total cyclic orders may be identified with $n$-cycles in $\symm_n$.
			\item[3.]
			Total cyclic orders 
			may be 
			geometrically visualized as
			$n$ dots in a directed cycle labeled 
			by $1,\dots,n$ with no repeats. These configurations are called {\em cyclic permutations},
			and will be used in the study of cyclic shuffles, see Figure~\ref{fig:cyc_shuffle}.
			
		\end{itemize}
	\end{remark}
	
	\begin{definition}
		A {\em toric $[\DAG{D}]$-partition} is a function $f: \{1,2,\ldots,n\} \rightarrow \{1,2,\ldots\}$ which
		is a $\DAG{D}'$-partition for at least one DAG $\DAG{D}'$ in $[\DAG{D}]$.
		Let  $\AAA^\tor([\DAG{D}])$ denote the set of all toric $[\DAG{D}]$-partitions. \end{definition}

	\begin{lemma}
		\label{cyclic-fundamental-lemma}
		For any DAG $\DAG{D}$, one has a decomposition of $\AAA^\tor([\DAG{D}])$ as the following disjoint union:
		\[
		\AAA^\tor([\DAG{D}]) 
		= \bigsqcup_{ [\DAG{w}] \in \LLL^\tor([\DAG{D}]) } \AAA^\tor([\DAG{w}]).
		\]
	\end{lemma}
	
	\begin{proof}
		The assertion about disjointness already follows,
		since by definition, one has 
		\begin{equation}
			\label{eq:toric-P-partitions-for-total-cyclic-orders}
			\AAA^\tor([\DAG{w}])  = \bigcup_{ w' \in [\DAG{w}] } \AAA(\DAG{w}'),
		\end{equation}
		and this union is disjoint by Lemma~\ref{fundamental-lemma-of-P-partitions}.
		For the union assertion, one has these equalities, justified below:
		$$
		\begin{aligned}
		\AAA^\tor([\DAG{D}]) 
		\overset{(i)}{=} \bigcup_{ \DAG{D}' \in [\DAG{D}]  } \AAA(\DAG{D}') 
		&\overset{(ii)}{=}\bigcup_{ \DAG{D}' \in [\DAG{D}]  } \quad  \bigcup_{ w' \in \LLL(\DAG{D}') } \AAA(\DAG{w}')\\
		&\overset{(iii)}{=}\bigcup_{ [\DAG{w}] \in \LLL^\tor([\DAG{D}])  } \quad  \bigcup_{ w' \in [\DAG{w}] } \AAA(\DAG{w}')\\
		&\overset{(iv)}{=}\bigcup_{ [\DAG{w}] \in \LLL^\tor([\DAG{D}])  } \quad \AAA^\tor([\DAG{w}])
		\end{aligned}
		$$
		Equality (i) is the definition of $\AAA^\tor([\DAG{D}])$,
		equality (ii) is Lemma~\ref{fundamental-lemma-of-P-partitions}, and equality (iv) uses \eqref{eq:toric-P-partitions-for-total-cyclic-orders} again.
		
		For equality (iii), one needs to show that $w' \in \LLL(\DAG{D}')$ for some
		$\DAG{D}' \in [\DAG{D}]$ if and only if one has $\DAG{w}' \in [\DAG{w}]$ for some $[\DAG{w}] \in \LLL^\tor([\DAG{D}])$.   The forward implication is straightforward: if $w' \in \LLL(\DAG{D}')$ for some
		$\DAG{D}' \in [\DAG{D}]$, then $\DAG{w}' \supseteq \DAG{D}'$, so that $[\DAG{w}']$ torically extends $[\DAG{D}']=[\DAG{D}]$.
		
		For the reverse implication, given $\DAG{w}' \in [\DAG{w}] \in \LLL^\tor([\DAG{D}])$, pick $\DAG{D}'' \in [\DAG{D}]$ and $\DAG{w}'' \in [\DAG{w}]$ with $\DAG{D}'' \subseteq \DAG{w}''$;  these exist because $[\DAG{w}] \in \LLL^\tor([\DAG{D}])$. Thus $[\DAG{w}'']=[\DAG{w}]=[\DAG{w}']$.  Now perform
		a sequence of flips that takes $\DAG{w}''$ to $\DAG{w}'$, and the same sequence of flips will
		take $\DAG{D}''$ to some $\DAG{D}' \subseteq \DAG{w}'$.  One then has 
		$\DAG{D}' \in [\DAG{D}'']=[\DAG{D}]$ and $w' \in \LLL(\DAG{D}')$, as desired.
	\end{proof}

	\subsection{Cyclic $P$-partition enumerators}
	
	\begin{definition}
		Given a toric poset $[\DAG{D}]$ on $\{1,2,\ldots,n\}$, 
		define its cyclic $P$-partition enumerator
		\[
		F^\cyc_{[\DAG{D}]}:=
		\sum_{f \in \AAA^\tor([\DAG{D}])} 
		x_{f(1)} x_{f(2)} \cdots x_{f(n)}.
		\]
	\end{definition}
	
	A special case of these enumerators are the fundamental cyclic quasi-symmetric functions, defined in Subsection~\ref{sec:FCQSym}.
	Recall the notations $P^{\cyc}_{n,J}$ and $F^\cyc_{n,J}$.
	
	\begin{proposition}
		If $w \in \symm_n$ has $\cDes(w)=J$, then the bijection 
		which sends $f:\{1,2,\ldots,n\} \rightarrow \PP$ to $v=(f(w_1),\ldots,f(w_n))$ in $\PP^n$
		restricts to a bijection
		\[
		\AAA^\tor([\DAG{w}])  \rightarrow 
		\{v\in \PP^n:\ \exists i \text{ such that } (v,i)\in P^{\cyc}_{n,J} \}.
		\]
		Consequently, 
		\[
		F^\cyc_{[\DAG{w}]} = F^\cyc_{n,J}.
		\]
		In particular, 
		this shows that $F^\cyc_{[\DAG{w}]}$ lies in $\cQSym^-$.
	\end{proposition}
	
	
	\begin{proof}
		By definition, $\AAA^\tor([\DAG{w}])=\bigsqcup_{w' \in [\DAG{w}]} \AAA(\DAG{w}')$.
		If  $w=(w_1,\ldots,w_n)$, then for each element
		$$
		w'=(w_i,w_{i+1},\ldots,w_n,w_1,w_2,\ldots,w_{i-1})
		$$
		of $[\DAG{w}]$, the bijection in the proposition sends $\AAA(\DAG{w}')$
		to 
		$
		\{v\in \PP^n:\ (v,i)\in P^{\cyc}_{n,J} \}.
		$
	\end{proof}

	An immediate consequence of 
	Lemma~\ref{cyclic-fundamental-lemma} is then the following.
	
	\begin{proposition}
		\label{cyclic-F-expansion}
		For any toric poset $[\DAG{D}]$,
		one has the following expansion
		\[
		F^\cyc_{[\DAG{D}]}
		=\sum_{[\DAG{w}] \in \LLL^\tor([\DAG{D}])}
		F^\cyc_{n,\cDes(w)}.
		\]
		Note that the sum is over classes $[\DAG{w}]$,
		not over permutations $w$, but $\cDes(w)$ depends
		only upon the class $[\DAG{w}]$.
		In particular, this shows that $F^\cyc_{[\DAG{D}]}$ lies in $\cQSym^-$.
	\end{proposition}
	
	We now use this fact to expand products
	of basis elements $\{F^\cyc_{n,J}\}$ back in the same
	basis.  The key notion is that of a cyclic shuffle
	of two total cyclic orders.
	
	\medskip
	
	First recall 
	the notion of shuffles of permutations.  From now on, we no longer require
	our DAGs, posets, and toric posets to have label set $\{1,2,\ldots,n\}$,
	but allow more general finite subsets of $\ZZ$ as labels.
	For a finite set $A$ of size $a$, let $\symm_A$ be 
	the set of all bijections $w : [a] \to A$,
	viewed as words $w = (w_1, \ldots, w_a)$.
	If $A = [a]$ then $\symm_A$ is of course the symmetric group $\symm_a$. 
	Elements of $\symm_A$ will be called {\em bijective words}, a formal extension of permutations.
	If $A$ is a set of integers, or any totally ordered set, define a {\em descent set} for each $w \in \symm_A$ by
	\[
	\Des(w) := \{i \in [a-1] \,:\, w(i) > w(i+1)\}.
	\]
	Given two bijective words $u = (u_1,\ldots,u_a)\in \symm_A$ and $v = (v_1,\ldots,v_b)\in \symm_B$, 
	where $A$ and $B$ are disjoint finite sets of integers, 
	a bijective word $w \in \symm_{A \sqcup B}$ is a {\em shuffle} of $u$ and $v$ 
	if $u$ and $v$ are subwords of $w$.
	Denote the set of all shuffles of $u$ and $v$ by 
	$u \shuffle v$.
	The following theorem is a straightforward generalization of~\cite[Ex.\ 7.93]{EC2}.
	
	\begin{theorem}
		\label{non-cyclic-product-of-F}
		Let $C = A \sqcup B$ be a disjoint union of finite sets of integers.
		For each $u \in \symm_A$ and $v \in \symm_B$, one has 
		the following expansion:
		\[
		F_{|A|,\Des(u)} \cdot F_{|B|,\Des(v)}
		= \sum_{w \in u\shuffle v } F_{|C|,\Des(w)}.
		\]
	\end{theorem}

	\begin{definition}
		Let $C = A \sqcup B$ be a 
		disjoint union of finite sets.
		Fix two total cyclic orders $[\DAG{u}]$ and $[\DAG{v}]$,
		with representatives $u = (u_1,\ldots,u_a)\in \symm_A$
		and $v = (v_1,\ldots,v_b)\in \symm_B$. 
		A total cyclic order $[\DAG{w}]$, with $w \in \symm_C$, is a {\em cyclic shuffle of $[\DAG{u}]$ and $[\DAG{v}]$}
		if there exists a representative $w' \in \symm_C$ of $[\DAG{w}]$
		which is 
		(equivalently, every representative of $[\DAG{w}]$ is) a shuffle of cyclic shifts of $u$ and $v$, namely,
		\[
		w' \in u' \shuffle v'
		\]
		for some cyclic shift $u'$ of $u$ and cyclic shift $v'$ of $v$.
		Here we have used the simple fact that if $u$ is a subword of $w$, then each cyclic
		shift of $w$ has some cyclic shift of $u$ as a subword.
		
		Denote the set of all cyclic shuffles of $[\DAG u]$ and $[\DAG v]$ by 
		$[\DAG u]\shuffle_\cyc [\DAG v]$.
	\end{definition}
	
	
	
	\begin{example}
		\label{cyclic-shuffle-example}
		Let $A=\{1,3,5,7,8\}$ and $B=\{2,4,6,9\}$ (so that $C=[9]$), and fix
		$u = (3,7,8,5,1) \in \symm_A$ and $v = (6,9,4,2) \in \symm_B$.
		An example of
		$[\DAG{w}] \in [\DAG{u}] \shuffle_{\cyc} [\DAG{v}]$
		is 
		$[(8,4,5,1,2,3,6,7,9)]$, since $w' = (1,2,3,6,7,9,8,4,5)$ is a shuffle of $(1,3,7,8,5)\in [\DAG{u}]$ and $(2,6,9,4)\in [\DAG{v}]$. See Figure~\ref{fig:cyc_shuffle}.
	\end{example}
	
	\begin{figure}
		\begin{center}
			
			\begin{tikzpicture}[scale=0.25]
			\draw[red] (0,0) circle (5);
			\draw[fill] (90:5) circle (.15); \draw[blue](90:6) node
			{$8$}; 
			\draw[fill] (50:5) circle (.15); \draw[orange](50:6) node
			{$4$}; 
			\draw[fill] (10:5) circle (.15); \draw[blue](10:6) node
			{$5$}; 
			\draw[fill] (330:5) circle (.15); \draw[blue](330:6) node
			{$1$}; 
			\draw[fill] (290:5) circle (.15);\draw[orange] (290:6) node
			{$2$};
			\draw[fill] (250:5) circle (.15);\draw[blue] (250:6) node
			{$3$};
			\draw[fill] (210:5) circle (.15);\draw[orange] (210:6) node
			{$6$};
			\draw[fill] (170:5) circle (.15);\draw[blue] (170:6) node
			{$7$};
			\draw[fill] (130:5) circle (.15);\draw[orange] (130:6) node
			{$9$};
			\draw[red] (180:5) node {$\wedge$};
			\draw[black] (0:9.5) node {$\in$};
			\end{tikzpicture}
			\ \ \ 
			\begin{tikzpicture}[scale=0.25]
			\draw[red] (0,0) circle (5);
			\draw[fill] (90:5) circle (.15); \draw[blue](90:6) node
			{$3$}; 
			\draw[fill] (18:5) circle (.15); \draw[blue](18:6) node
			{$7$}; 
			\draw[fill] (306:5) circle (.15);\draw[blue] (306:6) node
			{$8$};
			\draw[fill] (234:5) circle (.15);\draw[blue] (234:6) node
			{$5$};
			\draw[fill] (162:5) circle (.15);\draw[blue] (162:6) node
			{$1$};
			
			\draw[red] (180:5) node {$\wedge$};
			\end{tikzpicture}
			\ \ \
			\begin{tikzpicture}[scale=0.25]
			\draw[red] (0,0) circle (5);
			\draw[fill] (45:5) circle (.15); \draw[orange](54:6) node {$6$}; 
			
			\draw[fill] (315:5) circle (.15); \draw[orange](306:6) node {$9$}; 
			
			\draw[fill] (225:5) circle (.15);\draw[orange] (234:6) node {$4$};
			
			\draw[fill] (135:5) circle (.15);\draw[orange] (126:6) node {$2$};
			
			\draw[red] (180:5) node {$\wedge$};
			
			\draw[black] (180:10) node {$\shuffle_{\cyc}$}; 
			\end{tikzpicture}
		\end{center}
		\caption{$[(8,4,5,1,2,3,6,7,9)]\in [(3,7,8,5,1)]\shuffle_\cyc [(6,9,4,2)]$.}
		\label{fig:cyc_shuffle}
	\end{figure}
	
	We deduce the following cyclic analogue of Theorem~\ref{non-cyclic-product-of-F}.
	For completeness, define $\cDes(u) := \varnothing$ if $u$ is the unique element of $\symm_A$ when $|A|=1$ or $|A|=0$.
	Recall that, 
	by Remark~\ref{rem:cQSym_0}, $F^{\cyc}_{0,\varnothing} = 1$.
	
	\begin{theorem}
		\label{product-of-F-corollary}
		Let $C = A \sqcup B$ be a disjoint union of finite 
		sets of integers.
		For each $u \in \symm_A$ and $v \in \symm_B$, one has 
		the following expansion:
		\[
		F^\cyc_{|A|,\cDes(u)} \cdot F^\cyc_{|B|,\cDes(v)}
		= \sum_{[\DAG{w}] \in [\DAG{u}] \shuffle_{\cyc} [\DAG{v}] } F^\cyc_{|C|,\cDes(w)}.
		\]
	\end{theorem}
	
	
	\begin{proof}
		The claim clearly holds if either $A = \varnothing$ or $B = \varnothing$. We can therefore assume that $A$ and $B$ are nonempty.
		
		Consider the toric poset $[\DAG{D}]$ where $\DAG{D}$ is the disjoint union of the posets $\DAG{u}$ and $\DAG{v}$, and its associated
		cyclic quasi-symmetric function $F^\cyc_{[\DAG{D}]}$.
		E.g., if $u,v$ are as in Example~\ref{cyclic-shuffle-example}, then $[\DAG{D}]$ is represented by this $\DAG{D}$:
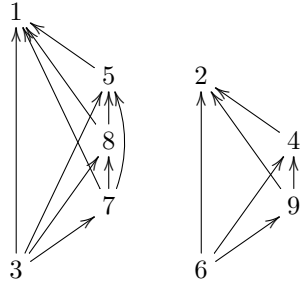
		\[
		\xymatrix@R=12pt{
			1&                                   &\\
			& 5 \ar[ul]                         &2&\\
			& 8 \ar[u] \ar[uul]                 & &4\ar[ul]\\
			& 7 \ar[u]  \ar[uuul] \ar@/_/[uu]   & &9\ar[u]\ar[uul]\\
			3\ar[ur] \ar[uur] \ar[uuur]\ar[uuuu]&&6\ar[ur]\ar[uuu] \ar[uur]&
		}
		\]
		On one hand, the definition of $F^\cyc_{[\DAG{D}]}$ lets one
		express it as a product as follows:
		\begin{equation}
			\label{disjoint-union-is-product}
			F^\cyc_{[\DAG{D}]} = 
			F^\cyc_{[\DAG{u}]}
			\cdot F^\cyc_{[\DAG{v}]}
			= F^\cyc_{|A|,\cDes(u)}
			\cdot F^\cyc_{|B|,\cDes(v)}.
		\end{equation}
		On the other hand, one can see from the definition that
		\[
		\LLL^\tor([\DAG{D}]) = [\DAG{u}] \shuffle_\cyc [\DAG{v}],
		\]
		and hence Proposition~\ref{cyclic-F-expansion}
		shows that 
		\begin{equation}
			\label{cyclic-shuffles-are-toric-extensions}
			F^\cyc_{[\DAG{D}]} 
			= \sum_{[\DAG{w}] \in \LLL^\tor([\DAG{D}])} F^\cyc_{[\DAG{w}]}
			=\sum_{[\DAG{w}] \in [\DAG{u}] \shuffle_\cyc [\DAG{v}]}
			F^\cyc_{|C|,\cDes(w)}.
		\end{equation}
		Equating the right hand sides of \eqref{disjoint-union-is-product} and \eqref{cyclic-shuffles-are-toric-extensions} gives the proposition.
	\end{proof}
	
	\begin{proposition}
		\label{number-of-cyclic-shuffles}
		Let $A$ and $B$ be disjoint nonempty sets of integers, of sizes $a$ and $b$ respectively.
		For each $u=(u_1,u_2,\ldots,u_a) \in \symm_A$ and $v=(v_1,v_2,\ldots,v_b) \in \symm_B$ there are
		$\frac{(a+b-1)!}{(a-1)!(b-1)!}$ cyclic shuffles 
		in $[\DAG{u}] \shuffle_{\cyc} [\DAG{v}]$.
	\end{proposition}
	\begin{proof}
		The number of such cyclic shuffles $[\DAG{w}]$ is
		$b$ times the number of ways to place the letters $v_1,v_2,\ldots,v_b$ in this order
		in the $a$ different slots between the following $u$-letters
		(zero, one, or more $v$-letters in each slot):
		\[
		u_1 \underline{\quad} u_2 \underline{\quad} u_3 \underline{\quad} \cdots \underline{\quad} u_a \underline{\quad} u_1
		\]
		This number is $b \cdot \binom{a+b-1}{b} = \frac{(a+b-1)!}{(a-1)!(b-1)!}$.
	\end{proof}
	
	\begin{example}
		Let's compute $F^\cyc_{3,\{1\}} \cdot F^\cyc_{2,\{1\}}$ using 
		Theorem~\ref{product-of-F-corollary}.  
		First pick $u \in \symm_{[3]}$ and $v \in \symm_{[5]\setminus [3]}$ such that $\cDes(u)=\{1\}$ and $\cDes(v)=\{1\}$,
		say $u=(3,1,2)$ and $v=(5,4)$.  
		Since $|A|=3$ and $|B|=2$, 
		Proposition~\ref{number-of-cyclic-shuffles} implies that there
		are $\frac{4!}{2! 1!}=12$ cyclic shuffles $[\DAG{w}]$ in
		$[\DAG{u}] \shuffle_\cyc [\DAG{v}]$. These cyclic shuffles are listed below:
		\[
		\begin{tabular}{|c|c|} \hline
		$w \in [\DAG{w}]$ with $w_1 = 3$ & $\cDes(w)$ \\ \hline\hline
		$(3,1,2,5,4)$ & $\{1,4,5\}$ \\ \hline
		$(3,1,2,4,5)$ & $\{1,5\}$ \\ \hline
		$(3,1,5,2,4)$ & $\{1,3,5\}$ \\ \hline
		$(3,1,4,2,5)$ & $\{1,3,5\}$ \\ \hline
		$(3,5,1,2,4)$ & $\{2,5\}$ \\ \hline
		$(3,4,1,2,5)$ & $\{2,5\}$ \\ \hline
		$(3,1,5,4,2)$ & $\{1,3,4\}$ \\ \hline
		$(3,1,4,5,2)$ & $\{1,4\}$ \\ \hline
		$(3,5,1,4,2)$ & $\{2,4\}$ \\ \hline
		$(3,4,1,5,2)$ & $\{2,4\}$ \\ \hline
		$(3,5,4,1,2)$ & $\{2,3\}$ \\ \hline
		$(3,4,5,1,2)$ & $\{3\}$ \\ \hline
		\end{tabular}
		\]
		Consequently,
		\[
		\begin{aligned}
		F^\cyc_{3,\{1\}} \cdot F^\cyc_{2,\{1\}}
		&= F^\cyc_{5,\{1,4,5\}}
		+ 2F^\cyc_{5,\{1,3,5\}}
		+ F^\cyc_{5,\{1,3,4\}}
		+ F^\cyc_{5,\{1,5\}}
		+ 2F^\cyc_{5,\{2,5\}}
		+ F^\cyc_{5,\{1,4\}}
		+ 2F^\cyc_{5,\{2,4\}}
		+ F^\cyc_{5,\{2,3\}}
		+ F^\cyc_{5,\{3\}} \\
		&= F^\cyc_{5,\{1,2,3\}}
		+ 3F^\cyc_{5,\{1,2,4\}}
		+ 2F^\cyc_{5,\{1,2\}}
		+ 5F^\cyc_{5,\{1,3\}}
		+ F^\cyc_{5,\{1\}}.
		\end{aligned}
		\]
	\end{example}
	
	
	
	\begin{proposition}
		\label{t:cQSym_is_a_ring}
		$\cQSym$ and $\cQSym^-$ are graded rings.
	\end{proposition}
	\begin{proof}
		%
		By Lemma~\ref{t:Fcyc_basis}(3), in order to show that $\cQSym$ is closed under multiplication it suffices to show that, 
		for any two subsets $J \subseteq [a]$ and $K \subseteq [b]$, the product $\hF^\cyc_{a,J} \cdot \hF^\cyc_{b,K}$ lies in $\cQSym_{a+b}$.
		Similarly, by 
		Definition~\ref{def:cQSym-},
		in order to show that $\cQSym^-$ is closed under multiplication it suffices to show that this product lies in $\cQSym^-_{a+b}$ whenever 
		either $\varnothing \subsetneq J \subsetneq [a]$ 
		or $(a,J) = (1, \varnothing)$,
		and similarly for $b$ and $K$.
		
		Let us start from the latter claim. 
		Indeed, any $\varnothing \subsetneq J \subsetneq [a]$ 
		(or $J = \varnothing$ for $a = 1$) 
		is the cyclic descent set of some $u \in \symm_a$, and similarly for $K$ and $v \in \symm_b$.
		By Theorem~\ref{product-of-F-corollary},
		$F^\cyc_{a,J} \cdot F^\cyc_{b,K}$ is a sum of terms of the form $F^{\cyc}_{a+b,L}$ with $\varnothing \subsetneq L \subsetneq [a+b]$, and therefore belongs to $\cQSym_{a+b}^-$.
		Normalizing, it follows that 
		\[
		\hF^\cyc_{a,J} \cdot \hF^\cyc_{b,K} \in \cQSym_{a+b}^- \otimes \,\QQ.
		\]
		Since the two factors lie in $\ZZ[[X]]:=\ZZ[[x_1,x_2,\ldots]]$, their product lies in
		$\ZZ[[X]] \cap \left( \cQSym^- \otimes \,\QQ \right) = \cQSym^-$.
		
		As for the somewhat larger $\cQSym$, it remains 
		to show that $\hF^\cyc_{a,[a]} \cdot f \in \cQSym$ for every $a \ge 1$ and $f \in \cQSym$;
		note that, because of the linear dependence in Lemma~\ref{t:Fcyc_basis}(2), there is no need to check $\hF^\cyc_{a,\varnothing} \cdot f$.
		Indeed, $\hF^\cyc_{a,[a]} = e_a$ and, for any sequence $i_1 < \ldots < i_t$ of indices and any sequence $m = (m_1, \ldots, m_t)$ of positive exponents,
		\[
		\left(
		\text{coefficient of }
		x_{i_1}^{m_1} \cdots x_{i_t}^{m_t}\text{ in }
		e_a \cdot f
		\right)
		=
		\sum_{\substack{A \subseteq \{i_1,\ldots,i_t\} \\ |A| = a}}
		\left( 
		\text{coefficient of } 
		\frac{x_{i_1}^{m_1} \cdots x_{i_t}^{m_t}}{\prod_{j \in A}{x_j}} 
		\text{ in } f 
		\right).
		\]
		Since $f$ is cyclic quasi-symmetric,
		this sum will be unchanged when we replace $m = (m_1,\ldots,m_t)$ by any cyclic shift $m' = (m'_1,\ldots,m'_t)$ and replace $i_1 < \cdots < i_t$ by any
		$i'_1 < \cdots < i'_t$.
	\end{proof}
	
	\begin{corollary}\label{t:nonnegative_structure_constants_minus}
		The structure constants of $\cQSym^-$, with respect to the normalized fundamental basis, are nonnegative integers.
	\end{corollary}
	\begin{proof}
		Nonnegativity follows from Theorem~\ref{product-of-F-corollary},
		and integrality follows from Proposition~\ref{t:cQSym_is_a_ring}.
	\end{proof}
	An analogous statement holds for $\cQSym$;
	see Proposition~\ref{t:nonnegative_structure_constants} below.

	\subsection{The involution $\omega$}
	
	Let $\pi_0$ be the longest permutation in the symmetric group $\symm_n$, defined by $\pi_0(i) := n+1-i$ $(1 \le i \le n)$. 
	Recall that the involution $\omega$ on  $\symm_n$ defined by $\omega(\pi) := \pi_0 \pi$ ($\forall \pi \in \symm_n)$ induces, via the mapping $\pi \mapsto \Des(\pi)$, the involutions (also denoted $\omega$) $J \mapsto [n-1] \setminus J$ on $2^{[n-1]}$, 
	$F_{n,J} \mapsto F_{n,[n-1] \setminus J}$ on $\QSym_n$. 
	In particular, the latter involution restricts to the involution $s_{\lambda} \mapsto s_{\lambda'}$ on $\Sym_n$
	\cite[\S7.6 and Theorem 7.14.5]{EC2}, where $\lambda'$ is the {\it conjugate} partition to $\lambda$.
	In fact, $\omega$ is a ring automorphism of $\QSym$; 
	see~\cite[Corollary 2.4]{MalvenutoReutenauer} and~\cite[Ex.~7.94a]{EC2}.
	
	\begin{lemma}\label{t:omega}
		$\cQSym_n$ and $\cQSym_n^-$ are invariant under the involution $\omega$. Explicitly,
		\[
		\omega \left( \hF^{\cyc}_{n,J} \right) = \hF^{\cyc}_{n,[n] \setminus J}
		\qquad (\forall J \subseteq [n]).
		\]
		The corresponding restrictions of $\omega$ are ring automorphisms of both $\cQSym$ and $\cQSym^-$.
	\end{lemma}
	\begin{proof}
		By Proposition~\ref{t:Fcyc_from_F}, for any $J \subseteq [n]$
		\[
		\omega \left( F^{\cyc}_{n,J} \right) 
		= \omega \left( \sum_{i \in [n]} F_{n, (J - i) \cap [n-1]} \right)
		= \sum_{i \in [n]} F_{n, [n-1] \setminus (J - i)} 
		= \sum_{i \in [n]} F_{n, (([n] \setminus J) - i) \cap [n-1]} = F^{\cyc}_{n,[n] \setminus J} \ .
		\]
		Since $d_J = d_{[n] \setminus J}$, the same claim holds also for the normalized functions $\hF^{\cyc}_{n,J}$.
		
		Finally, since $\omega$ is an automorphism of $\QSym$, it is also an automorphism of each of its invariant subrings $\cQSym$ and $\cQSym^-$.
	\end{proof}

	\section{Expansion of Schur functions in terms of fundamental cyclic quasi-symmetric functions}\label{sec:expansion}
	
	Recall that
	$\QSym$, $\cQSym$ and the set $\Sym$ of symmetric functions 
	(sometimes denoted $\Lambda$) are graded rings satisfying
	\[
	\Sym \subseteq \cQSym \subseteq \QSym.
	\]
	
	In this section we prove the following theorem.
	
	\begin{theorem}\label{conj:QS11}
		For every skew shape $\lambda/\mu$ which is not a connected ribbon, 
		all the coefficients in the expansion of the skew Schur function $s_{\lambda/\mu}$
		in terms of normalized fundamental cyclic quasi-symmetric functions are nonnegative integers. 
	\end{theorem}
	
	Theorem~\ref{conj:QS11} follows from 
	Corollary~\ref{t:Schur_in_hFcyc} 
	below.
	The {\em cyclic descent map} on SYT of a given shape plays a key role in the proof; 
	let us recall the relevant definition and main result from \cite{ARR_CDes}.
	
	\begin{definition}[{\cite[Definition 2.1]{ARR_CDes}}]\label{def:cDes}
		Let $\TTT$ be a finite set, equipped with an arbitrary map (called {\em descent map})
		$\Des: \TTT \longrightarrow 2^{[n-1]}$, where $n > 1$. 
		A {\em cyclic extension} of $\Des$ is
		a pair $(\cDes,p)$, where 
		$\cDes: \TTT \longrightarrow 2^{[n]}$ is a map 
		and $p: \TTT \longrightarrow \TTT$ is a bijection,
		satisfying the following axioms:  for all $T$ in $\TTT$:
		\[
		\begin{array}{rl}
		\text{(extension)}   & \cDes(T) \cap [n-1] = \Des(T),\\
		\text{(equivariance)}& \cDes(p(T))  = 1 + \cDes(T) \,\, \bmod{n},\\
		\text{(non-Escher)}  & \varnothing \subsetneq \cDes(T) \subsetneq [n].\\
		\end{array}
		\]
	\end{definition}
	
	\begin{example}\label{ex:Cellini}
		Let $\TTT$ be $\symm_n$, the symmetric group on $n$ letters.
		The classical {\em descent set} of a permutation $\pi = (\pi_1, \ldots, \pi_n)\in \symm_n$ is
		\[
		\Des(\pi) := \{1 \le i \le n-1 \,:\, \pi_i > \pi_{i+1} \}
		\quad \subseteq [n-1],
		\]
		where $[m]:=\{1,2,\ldots,m\}$.  Its
		{\em cyclic descent set}, as defined by Cellini~\cite{Cellini}, is
		\begin{equation}\label{def:Cellini-dDes}
			\cDes(\pi) := \{1 \leq i \leq n \,:\, \pi_i > \pi_{i+1} \}
			\quad \subseteq [n],
		\end{equation}
		with the convention $\pi_{n+1}:=\pi_1$.
		The pair $(\cDes,p)$ satisfies the axioms of Definition~\ref{def:cDes} for $p: \symm_n \longrightarrow \symm_n$ defined by 
		$p(\pi) := \pi \circ c$, where $c$ is the cyclic shift mapping $i$ to $i-1 \pmod n$.
	\end{example}
	
	\begin{remark}\label{rem:non-Escher}
		The term {\em non-Escher} in Definition~\ref{def:cDes} refers to M.\ C.\ Escher's drawing ``Ascending and Descending", which paradoxically
		depicts the impossible cases $\cDes(\pi) = \varnothing$ and $\cDes(\pi) = [n]$
		in Example~\ref{ex:Cellini}.
	\end{remark}
	
	
	\medskip
	There is also an established notion (see, e.g., \cite[p.~361]{EC2}) 
	of a descent set for a {\em standard (Young) tableau} $T$ of skew shape $\lambda/\mu$ (and size $n$):
	\[
	\Des(T) := \{1 \le i \le n-1 \,:\,
	i+1 \text{ appears in a lower row of }T\text{ than }i\}
	\quad \subseteq [n-1].
	\]
	For example, the following standard Young tableau $T$ of shape $\lambda/\mu=(4,3,2)/(1,1)$ 
	has $\Des(T)=\{2,3,5\}$:
	\[
	\large\young(:127,:35,46)
	\]
	For the special case of {\em rectangular} shapes,
	Rhoades~\cite{Rhoades} 
	constructed a cyclic extension $(\cDes,p)$ of $\Des$ satisfying the axioms of Definition~\ref{def:cDes},
	where $p$ acts on SYT via Sch\"utzenberger's {\em jeu-de-taquin promotion} operator. 
	For almost all skew shapes there is a general existence result, as follows.
	
	\begin{theorem}[{\cite[Theorem 1.1]{ARR_CDes}}]\label{t:cyclic_extension}
		Let $\lambda/\mu$ be a skew shape with $n$ cells.
		The descent map $\Des$ on $\SYT(\lambda/\mu)$ has a cyclic extension $(\cDes,p)$ if and only if $\lambda/\mu$ is not a connected ribbon.
		Furthermore, for all $J \subseteq [n]$, all such cyclic extensions share the same cardinalities $\#{\cDes^{-1}(J)}$.
	\end{theorem}
	
	
	We shall now provide, in Theorem~\ref{conj:QS1} and Corollary~\ref{t:Schur_in_hFcyc},
	cyclic analogues of the classical result 
	\cite[Theorem 7.19.7]{EC2} (first proved in~\cite[Theorem 7]{Gessel}).
	
	\begin{theorem}\label{conj:QS1}
		For every skew shape $\lambda/\mu$ of size $n$ that is not a connected ribbon, and for any 
		cyclic extension $(\cDes,p)$ of $\Des$ 
		on $\SYT(\lambda/\mu)$,
		\[
		n s_{\lambda/\mu}
		= \sum_{T \in \SYT(\lambda/\mu)} F^{\cyc}_{n,\cDes(T)}
		\]
		where $s_{\lambda/\mu}$ is the Schur function indexed by $\lambda/\mu$.
	\end{theorem}	
	
	
	\begin{proof}
		Denoting
		\[
		m^{\cyc}(J) := \# \{T \in \SYT(\lambda/\mu) \,:\, \cDes(T) = J\}
		\qquad (\forall J \subseteq [n]),
		\]
		it follows from Lemma~\ref{t:Fcyc_vs_Mcyc} that
		\[
		\sum_{T \in \SYT(\lambda/\mu)} F^{\cyc}_{n, \cDes(T)}
		= \sum_{J \subseteq [n]} m^{\cyc}(J) F^{\cyc}_{n, J}
		= \sum_{J \subseteq [n]} m^{\cyc}(J) 
		\sum_{i \in [n]} F_{n, (J - i) \cap [n-1]}.
		\]
		The equivariance axiom from Definition~\ref{def:cDes} translates to
		cyclic invariance of the multiplicities $m^{\cyc}$:
		\[
		m^{\cyc}(J - i) = m^{\cyc}(J) 
		\qquad (\forall J \subseteq [n],\, i \in [n]).
		\]
		This, in turn, implies that
		\[
		\sum_{T \in \SYT(\lambda/\mu)} F^{\cyc}_{n, \cDes(T)}
		= \sum_{J \subseteq [n]} \sum_{i \in [n]} 
		m^{\cyc}(J - i) F_{n, (J - i) \cap [n-1]}
		= n \sum_{J \subseteq [n]} 
		m^{\cyc}(J) F_{n, J \cap [n-1]}.
		\]
		Distinguishing the cases $n \not\in J$ and $n \in J$, 
		and changing notation accordingly, we can write
		\[
		\sum_{T \in \SYT(\lambda/\mu)} F^{\cyc}_{n, \cDes(T)}
		= n \sum_{J \subseteq [n-1]} 
		\left( m^{\cyc}(J) + m^{\cyc}(J \sqcup \{n\}) \right) F_{n, J}.
		\]
		Denoting
		\[
		m(J) := \# \{T \in \SYT(\lambda/\mu) \,:\, \Des(T) = J\}
		\qquad (\forall J \subseteq [n-1]),
		\]
		the extension axiom from Definition~\ref{def:cDes} is equivalent to
		\[
		m(J) = m^{\cyc}(J) + m^{\cyc}(J \sqcup \{n\}) 
		\qquad (\forall J \subseteq [n-1]),
		\]
		which now gives
		\[
		\sum_{T \in \SYT(\lambda/\mu)} F^{\cyc}_{n, \cDes(T)}
		= n \sum_{J \subseteq [n-1]} m(J) F_{n, J}.
		\]
		The classical formula of Gessel~\cite[Theorem 7]{Gessel}
		\[
		s_{\lambda/\mu}
		= \sum_{T \in \SYT(\lambda/\mu)} F_{n, \Des(T)}
		= \sum_{J \subseteq [n-1]} m(J) F_{n, J}
		\]
		now completes the proof.
	\end{proof}
	
	The following corollary shows that the coefficients in the expression of a skew Schur function in terms of the basis of {\em normalized} fundamental cyclic quasi-symmetric functions are equal to the corresponding fiber sizes of (any) cyclic descent map.
	
	\begin{corollary}\label{t:Schur_in_hFcyc}
		Under the assumptions of Theorem~\ref{conj:QS1},
		\[
		s_{\lambda/\mu}
		= \sum_{A \in c2_{0,n}^{[n]}} m^{\cyc}(A) \,\hF^{\cyc}_{n,A}
		\]
		where
		\[
		m^{\cyc}(A) := m^{\cyc}(J)
		= \# \{T \in \SYT(\lambda/\mu) \,:\, \cDes(T) =J\}
		\qquad \left( \forall J \in A \in c2_{0,n}^{[n]} \right).
		\]
		This number is independent of the choice of $J$ in $A$. In fact,
		\[
		m^{\cyc}(A) 
		= \frac{1}{\# A} \sum_{J \in A} m^{\cyc}(J)
		= \frac{1}{\# A} \cdot \# \{T \in \SYT(\lambda/\mu) \,:\, \cDes(T) \in A\}.
		\]
	\end{corollary}
	\begin{proof}
		By Theorem~\ref{conj:QS1} and the above notations,
		using the non-Escher axiom from Definition~\ref{def:cDes},
		\[
		s_{\lambda/\mu}
		= \frac{1}{n} \sum_{T \in \SYT(\lambda/\mu)} F^{\cyc}_{n,\cDes(T)} 
		= \frac{1}{n} \sum_{\varnothing \subsetneq J \subsetneq [n]} m^{\cyc}(J) \,F^{\cyc}_{n,J}
		= \frac{1}{n} \sum_{A \in c2_{0,n}^{[n]}} \#A \cdot m^{\cyc}(A) \,F^{\cyc}_{n,A}.
		\]
		By Definition~\ref{def:Fcyc_normalized}
		and the equality $\#A \cdot d_A = n$,
		\[
		s_{\lambda/\mu}
		= \frac{1}{n} \sum_{A \in c2_{0,n}^{[n]}} \#A \cdot d_A \cdot m^{\cyc}(A) \,\hF^{\cyc}_{n,A}
		= \sum_{A \in c2_{0,n}^{[n]}} m^{\cyc}(A) \,\hF^{\cyc}_{n,A},
		\]
		as claimed.
	\end{proof}
	
	
	
	\begin{remark}\label{rem:versus}
		For every finite set $\TTT$ with a descent map
		$\Des: \TTT \longrightarrow 2^{[n-1]}$ 
		which has a cyclic extension $(\cDes,p)$, the following holds:
		\begin{equation}\label{eq:versus}
			\sum\limits_{T\in \TTT} \hF^{\cyc}_{n,\cDes(T)}=
			n \sum\limits_{T\in \TTT} F^{\cyc}_{n,\cDes(T)}
			=\sum\limits_{T\in \TTT} F_{n,\Des(T)}.
		\end{equation}
		The first equality follows from the proof of Corollary~\ref{t:Schur_in_hFcyc}.
		The second equality follows from the proof of Theorem~\ref{conj:QS1}.
	\end{remark}
	
	\medskip
	
	If $\lambda/\mu$ is a connected ribbon then a corresponding cyclic descent map $\cDes$ does not exist, by Theorem~\ref{t:cyclic_extension}.
	Nevertheless, one can combine the dual Jacobi-Trudi identity \cite[Corollary 7.16.2]{EC2} with 
	Proposition~\ref{t:cQSym_is_a_ring}
	to get an expansion of $s_{\lambda/\mu}$ as a linear combination of $\hF^{\cyc}_{n,A}$ $(A \in \cnes)$ with (not necessarily positive) integer coefficients. 
	In the special case that $\lambda/\mu$ is a hook, this result can be stated very explicitly using a simpler approach.
	
	\begin{proposition}\label{prop:QS_hook}
		For every $0 \le k \le n-1$,
		\[
		n s_{(n-k,1^k)} 
		= \sum_{\substack{J\subseteq [n] \\ |J|>k}} (-1)^{|J|-k-1} F^{\cyc}_{n,J}
		\]
		and
		\[
		s_{(n-k,1^k)} 
		= \sum_{\substack{A \in \cnes \\ r(A)>k}} (-1)^{r(A)-k-1} \hF^{\cyc}_{n,A}.
		\]
		Note that one of the summands is always
		$\hF^{\cyc}_{n,\{[n]\}} \not\in \cQSym^-$.
	\end{proposition}
	
	
	\begin{proof}
		First recall from~\cite[Theorem 3.11(a)]{AER} 
		that for every 
		$1 \le k \le n-1$, the multiset of cyclic descent sets of SYT of shape $(n-k+1,1^k)/(1) = (1^k) \oplus (n-k)$ is the set of all subsets of $[n]$ of size $k$, each of them appearing exactly once; namely,
		\begin{equation}\label{eq:oplus}
			\{\{\cDes(T) \,:\, T\in \SYT((1^k) \oplus (n-k)) \}\}
			= \{J\subseteq [n] \,:\, |J|=k\}.
		\end{equation}
		Equivalently, in the notation of the proof of Theorem~\ref{conj:QS1}, for this shape:
		\[
		m^{\cyc}(J) =
		\begin{cases}
		1,& \text{if } |J| = k; \\
		0,& \text{otherwise}
		\end{cases}
		\qquad (\forall J \subseteq [n]).
		\]
		Thus, by Theorem~\ref{conj:QS1},
		\begin{equation}
			\label{eq:QS}
			n s_{(1^k) \oplus (n-k)} 
			= \sum_{\substack{J\subseteq [n] \\ |J|=k}} F^{\cyc}_{n,J}.
		\end{equation}
		The above reasoning is valid for $1 \le k \le n-1$, but Formula~\eqref{eq:QS} actually holds also for $k = 0$  and $k = n$, by Lemma~\ref{t:Fcyc_properties}(1). (Theorem~\ref{conj:QS1} cannot be used in these cases, since the corresponding shapes $(n)$ and $(1^n)$ are connected ribbons.)
		
		On the other hand, by 
		the definition of (skew) Schur functions in terms of semistandard tableaux,
		\[
		s_{(1^k) \oplus (n-k)} = s_{(n-k,1^k)}+s_{(n-k+1,1^{k-1})}
		\qquad (1 \le k \le n-1).
		\]
		Thus
		\[
		s_{(n-k,1^k)} 
		= \sum_{j=k+1}^{n} (-1)^{j-k-1} s_{(1^j) \oplus (n-j)}
		\qquad (0 \le k \le n-1).
		\]
		Applying Equation~\eqref{eq:QS} to the RHS completes the proof.
	\end{proof}	
	
	
	
	
	
	
	
	
	
	We are now in a position to extend Corollary~\ref{t:nonnegative_structure_constants_minus} to $\cQSym$.
	\begin{proposition}\label{t:nonnegative_structure_constants}
		The structure constants of $\cQSym$, with respect to the normalized fundamental basis, are nonnegative integers.
	\end{proposition}
	\begin{proof}
		We already know from Proposition~\ref{t:cQSym_is_a_ring} that the structure constants are integers.
		Corollary~\ref{t:nonnegative_structure_constants_minus} shows that $\hF^{\cyc}_{a,A} \cdot \hF^{\cyc}_{b,B}$ is a nonnegative linear combination of basis elements $\hF^{\cyc}_{a+b,C}$ $(C \in c2_{0,a+b}^{[a+b]})$ whenever $A \in c2_{0,a}^{[a]}$ and $B \in c2_{0,b}^{[b]}$. 
		If $A = \{[a]\}$ and $B = \{[b]\}$ then, by Corollary~\ref{t:Schur_in_hFcyc},
		\[
		\hF^{\cyc}_{a,\{[a]\}} \cdot \hF^{\cyc}_{b,\{[b]\}} 
		= e_a \cdot e_b 
		= s_{(1^a) \oplus (1^b)}
		= \sum_{C \in c2_{0,a+b}^{[a+b]}} m^{\cyc}(C) \,\hF^{\cyc}_{a+b,C}
		\]
		where the skew shape $(1^a) \oplus (1^b)$,
		consisting of two disconnected columns $(1^a)$ and $(1^b)$, 
		is not a connected ribbon and thus the coefficients
		\[
		m^{\cyc}(C) = \frac{1}{\# C} \cdot \# \{T \in \SYT((1^a) \oplus (1^b)) \,:\, \cDes(T) \in C\}
		\]
		are nonnegative.
		It remains to show nonnegativity when $A = \{[a]\}$ but $B \in c2_{0,b}^{[b]}$. In that case, choose a bijective word $w_0 \in \symm_{\{a+1, \ldots, a+b\}}$ such that $\cDes(w_0) \in B$. Then, by Proposition~\ref{t:Fcyc_from_F},
		\[
		\hF^{\cyc}_{a,\{[a]\}} \cdot \hF^{\cyc}_{b,B} 
		= F_{a,[a-1]} \cdot \frac{1}{d_B} \sum_{i \in [b]} F_{b, (B-i) \cap [b-1]} 
		= F_{a,[a-1]} \cdot \frac{1}{d_B} \sum_{i \in [b]} F_{b,\Des(w_0 c^i)} 
		\]
		where $c \in \symm_b$ maps each $i$ to $i+1 \pmod b$.
		Thus, according to the classical Theorem~\ref{non-cyclic-product-of-F},
		\[
		\hF^{\cyc}_{a,\{[a]\}} \cdot \hF^{\cyc}_{b,B} 
		= \frac{1}{d_B} \sum_{i \in [b]} \sum_{w \in (a,a-1,\ldots,1) \shuffle w_0 c^i} F_{a+b,\Des(w)}.
		\]
		Denoting
		\[
		W := \bigsqcup_{i \in [b]} (a,a-1,\ldots,1) \shuffle w_0 c^i,
		\]
		it suffices to show that the set of bijective words $W$, with the usual descent map $\Des : W \to 2^{[a+b-1]}$, has a cyclic extension in the sense of Definition~\ref{def:cDes}.
		
		Indeed, define a map $\cDes^*: W \to 2^{[a+b]}$ as a slight deformation of Cellini's $\cDes$, defined in~\eqref{def:Cellini-dDes}:
		\[
		\cDes^*(w) := 
		\begin{cases}
		\Des(w) \sqcup \{a+b\},& \text{if } w(a+b) > w(1) \text{ or } w(1) = a; \\
		\Des(w),& \text{otherwise}  
		\end{cases}
		\qquad (\forall w \in W).
		\]
		The map $\cDes^*$ is non-Escher:
		$\cDes^*(w) = \varnothing$ implies that $\Des(w) = \varnothing$, namely that $w$ is the identity permutation, and in particular has the prefix $(1, \ldots, a-1, a)$ implying $w \not\in W$; unless $a = 1$, in which case $w(1) = a$ and thus $a+b \in \cDes^*(w)$.
		It is also easy to see that $\cDes^*(w) \ne [a+b]$ for any $w \in W$.
		
		Define $p(w)$ as the word obtained from $w$ by cyclically shifting the letters which are bigger than $a$ one position to the right modulo $a+b$, and inserting 
		the letters $a,a-1,\dots,1$ in the remaining positions in decreasing order.
		One can verify that
		$p: W \to W$ is a bijection which shifts $\cDes^*$ cyclically, namely 
		$\cDes^*(p(w)) = 1 + \cDes^*(w) \pmod {a+b}$ for any $w \in W$.
		For example, if $a=4$, $b=3$, and $w=43 75 2 6 1$ then 
		$p(w)= 432 75 1 6$ and $p^2(w)= 6 432 751 $,
		yielding $\cDes^*(w) = \{1,3,4,6,7\}$, $\cDes^*(p(w)) = \{1,2,4,5,7\}$, and $\cDes^*(p^2(w)) = \{1,2,3,5,6\}$.
		
		One concludes that
		\[
		\sum_{w \in W} F_{a+b,\Des(w)}
		= \frac{1}{a+b} \sum_{i \in [a+b]} \sum_{w \in W} F_{a+b,\Des(p^i(w))}
		= \frac{1}{a+b} \sum_{w \in W} F^\cyc_{a+b,\cDes^*(w)},
		\]
		completing the proof.
	\end{proof}
	
	Recall from~\cite{ARR_CDes} the {\em cyclic ribbon Schur function}, defined as
	\[
	\cs_{\cc(J)} := \sum_{\varnothing \neq I \subseteq J}
	(-1)^{\#(J \setminus I)} h_{\cc(I)} \in \Sym_n
	\qquad (\forall \varnothing \ne J \subseteq [n]).
	\]
	The following cyclic analogue of~\cite[Theorem 3]{Gessel} holds.
	
	\begin{lemma}\label{lem:f}
		For every (homogeneous of degree $n$) symmetric function $f \in \Sym_n$, the unique expansion
		\[
		f = \sum_{A \in \cnes} c_A \hF^{\cyc}_{n,A}
		\] 
		satisfies, for each representative $J$ of each orbit $A \in \cnes$:
		\[
		c_A = \langle f, \tilde s_{\cc(J)} \rangle,
		\]
		where $\langle \cdot, \cdot \rangle$ is the usual inner product of symmetric functions.
	\end{lemma}
	
	\begin{proof} 
		
		Write $f$ in the monomial basis of $\Sym_n$:
		\[
		f = \sum_{\lambda \vdash n} \langle f, h_{\lambda} \rangle m_{\lambda}.
		\]
		The monomial symmetric function
		has monomial quasisymmetric function expansion
		$$
		m_\lambda=\sum_{(m_1,\ldots,m_t)} M_{(m_1,\ldots,m_t)}
		$$
		in which the sum
		is over all compositions $(m_1,\ldots,m_t)$ that reorder to the partition $\lambda$.  Therefore one has
		\[
		m_{\lambda} 
		= \sum_{A\,:\, par(A) = \lambda} \hM^{\cyc}_{n,A} 
		= n^{-1} \sum_{J \,:\, par(J) = \lambda} M^{\cyc}_{n,J}
		\]
		where the second sum is over all the subsets $\varnothing \ne J \subseteq [n]$ such that the corresponding cyclic composition $\cc(J)$ can be reordered to the partition $\lambda$
		(and the first sum is over all cyclic orbits $A$ of such subsets). 
		Thus
		\[
		\begin{aligned}
		f 
		&= n^{-1} \sum_{\lambda \vdash n} \langle f, h_{\lambda} \rangle
		\sum_{J \,:\, par(J) = \lambda} M^{\cyc}_{n,J} \\
		&= n^{-1} \sum_{\varnothing \ne J \subseteq [n]} \langle f, h_{\cc(J)} \rangle M^{\cyc}_{n,J} \\
		&= n^{-1} \sum_{\varnothing \ne J \subseteq [n]} \langle f, h_{\cc(J)} \rangle 
		\sum_{K \supseteq J} (-1)^{\#(K \setminus J)} F^{\cyc}_{n,K} \\
		&= n^{-1} \sum_{\varnothing \ne K \subseteq [n]} F^{\cyc}_{n,K} 
		\sum_{\varnothing \ne J \subseteq K} (-1)^{\#(K \setminus J)} \langle f, h_{\cc(J)} \rangle \\
		&= n^{-1} \sum_{\varnothing \ne K \subseteq [n]} F^{\cyc}_{n,K} 
		\langle f, \cs_{\cc(K)} \rangle \\
		&= \sum_{A \in \cnes} \hF^{\cyc}_{n,A} 
		\langle f, \cs_{\cc(A)} \rangle.
		\end{aligned}
		\]
		In the last expression, $\cc(A)$ is shorthand for $\cc(K)$ for an arbitrary $K \in A$.
	\end{proof}	
	
	\begin{corollary}\label{cor:integrality}
		Let $f$ be a symmetric function, homogeneous of degree $n$.
		The coefficients in the expansion of $f$ as a linear combination of 
		$\left\{ \hF^{\cyc}_{n,A} \,:\, A \in \cnes \right\}$ 
		are integral
		if and only if the coefficients in its expansion as a linear combination of Schur functions are integral.
	\end{corollary}
	
	\begin{proof}
		Let $c_A(f)$ denote the coefficient of $\hF^{\cyc}_{n,A}$ in the expansion of $f$ in the above normalized fundamental basis.
		Clearly,
		\[
		c_A(f) = \sum_{\lambda \vdash n} \langle f, s_\lambda \rangle\, c_A(s_\lambda)
		\qquad (\forall A \in \cnes).
		\]
		By 
		Corollary~\ref{t:Schur_in_hFcyc} and Proposition~\ref{prop:QS_hook}, for every partition $\lambda \vdash n$ and $A \in \cnes$
		\[
		c_A(s_\lambda) \in \ZZ.
		\]  
		We conclude that if  the coefficients in the expansion of $f$ as a linear combination of Schur functions are integral then so are the coefficients in its expansion in the normalized fundamental basis.
		
		For the opposite direction,
		assume that all the coefficients $c_A(f)$ are integral. 
		Then, by Lemma~\ref{lem:f},
		\[
		\langle f, \cs_{\cc(J)} \rangle \in \ZZ 
		\qquad (\forall\ \varnothing \ne J \subseteq [n]).
		\]
		The definition of $\cs_{\cc(J)}$ as an alternating sum of complete homogeneous symmetric functions $h_{\cc(I)}$ is equivalent, by the Principle of Inclusion and Exclusion, to
		\[
		h_{\cc(J)} = \sum_{\varnothing \neq I \subseteq J} \cs_{\cc(I)}
		\qquad (\forall\ \varnothing \ne J \subseteq [n]),
		\]
		and therefore
		\[
		\langle f, h_{\cc(J)} \rangle \in \ZZ 
		\qquad (\forall\ \varnothing \ne J \subseteq [n]).
		\]
		By the Jacobi-Trudi identity~\cite[Theorem 7.16.1]{EC2}, any Schur function $s_\lambda$ is an alternating sum of complete homogeneous symmetric functions, so that
		\[
		\langle f, s_\lambda \rangle
		\in \ZZ 
		\qquad (\forall \lambda\vdash n),\]
		as claimed.
	\end{proof}
	
	Another application is the following.
	
	\begin{corollary}
		Let  $\TTT$ be a finite set with a descent map
		$\Des: \TTT \longrightarrow 2^{[n-1]}$ 
		which has a cyclic extension $(\cDes, p)$. If
		\[
		\Q(\TTT):=\sum\limits_{T\in \TTT} F_{n,\Des(T)}
		\]
		is symmetric then the fiber sizes of the cyclic descent map $\cDes$ satisfy
		\[
		\# \{T \in \TTT \,:\, \cDes(T) =J\}=\langle \Q(\TTT), \tilde s_{\cc(J)} \rangle
		\qquad \left( \forall J \in 2_{0,n}^{[n]} \right).
		\]
	\end{corollary}
	
	\begin{proof} Combine Lemma~\ref{lem:f} with 
		Remark~\ref{rem:versus}.
	\end{proof}
	
	\begin{remark}\label{rem_Postnikov} 
		By Lemma~\ref{lem:f} together with~\cite[Proposition 3.14]{ARR_CDes},  for every non-hook shape $\lambda$, the coefficient of $\hF^{cyc}_{n,[J]}$ in $s_\lambda$ is equal to the coefficient of $s_\lambda$ in the Schur expansion of Postnikov's toric Schur function $s_{\mu(J)/1/\mu(J)}$.
		It follows from Postnikov's result~\cite[Theorem 5.3]{Postnikov} 
		that these coefficients are equal to certain Gromov-Witten invariants. 
	\end{remark}

	\section{Enumerative applications}\label{sec:enumeration}
	
	\subsection{Distribution of cyclic descents over 
		cyclic shuffles}
	
	The  following classical result is due to Stanley. 
	
	\begin{proposition}[{\cite[Ex.~3.161]{EC1}}; see also {\cite[section 2.4]{GesselZhuang}} for a combinatorial proof.]
		\label{Stanley_lemma}
		Let $A$ and $B$ be two disjoint sets of integers.
		For each  $u \in \symm_A$ and $v \in \symm_B$,
		the distribution of the descent set over all shuffles of $u$ and $v$ depends only on $\Des(u)$ and $\Des(v)$.
	\end{proposition}
	
	We have the following cyclic analogue. Recall that for every 
	$u \in \symm_A$,
	$\cDes([\DAG u])$ is defined up to cyclic rotation.
	
	\begin{proposition}
		Let $A$ and $B$ be two disjoint sets of integers.
		For each  $u \in \symm_A$ and $v \in \symm_B$,
		the distribution of the cyclic descent set over all cyclic shuffles of $[\DAG u]$ and $[\DAG v]$ depends only on $\cDes([\DAG u])$ and $\cDes([\DAG v])$.
	\end{proposition}
	
	\begin{proof} This is an immediate consequence of Theorem~\ref{product-of-F-corollary}.
	\end{proof}
	
	\medskip
	
	Consider now $\ZZ[[q]]$, the ring of formal power series in $q$, as an abelian group under its usual addition, 
	but define a new product by
	\[
	q^i \odot q^j := q^{\max(i,j)},
	\]
	extended to satisfy the distributive law, even for infinite sums.
	We obtain a (commutative and associative) ring, to be denoted $\ZZ[[q]]_{\odot}$.
	
	Consider also the multivariate power series ring $\ZZ[[\xx]] = \ZZ[[x_1, x_2, \ldots]]$, with the usual product, and its subring $\ZZ[[\xx]]_\bded$ consisting of bounded-degree power series.
	Define a map $\Psi: \ZZ[[\xx]]_\bded \to \ZZ[[q]]_{\odot}$ by
	\[
	\Psi(x_{i_1}^{m_1} \cdots x_{i_k}^{m_k}) := q^{i_k}
	\qquad (k > 0,\, i_1 < \ldots < i_k,\, m_1, \ldots, m_k > 0)
	\]
	and
	\[
	\Psi(1) := 1,
	\]
	extended linearly, even for infinite sums.
	
	\begin{obs}
		$\Psi$ is a ring ($\ZZ$-algebra) homomorphism.
	\end{obs}
	
	We now compute the images, under $\Psi$, of various quasi-symmetric functions. All explicit formulas will be stated in $\ZZ[[q]]$, with the standard multiplication.
	
	\begin{lemma}\label{t:Psi_M_and_F}
		For any positive integer $n$ and subset $J \subseteq [n-1]$,
		\[
		\Psi(M_{n,J}) 
		= \left( \frac{q}{1-q} \right)^{|J|+1}
		= (1-q) \sum_r \binom{r}{|J|+1} q^r
		\]
		and
		\[
		\Psi(F_{n,J}) 
		= \frac{q^{\,|J|+1}}{(1-q)^{n}}
		= (1-q) \sum_r \binom{r+n-|J|-1}{n} q^r.
		\]
	\end{lemma}
	\begin{proof}
		Let $(m_0, \ldots, m_j) = \co(J)$  be the composition of $n$ corresponding to $J \subseteq [n-1]$, with $j = |J|$. Then
		\[
		\Psi(M_{n,J}) 
		= \sum_{i_0 < \ldots < i_j} \Psi(x_{i_0}^{m_0} \cdots x_{i_j}^{m_j})
		= \sum_{i_0 < \ldots < i_j} q^{i_j}
		= \sum_{r \ge j+1} \binom{r-1}{j} q^r
		= \left( \frac{q}{1-q} \right)^{j+1}
		\]
		and, consequently,
		\[
		\Psi(F_{n,J}) 
		= \sum_{J \subseteq K \subseteq [n-1]} \Psi(M_{n,K})
		= \sum_{k=j}^{n-1} \binom{n-1-j}{k-j} \left( \frac{q}{1-q} \right)^{k+1}
		= \left( \frac{q}{1-q} \right)^{j+1} \left( 1 + \frac{q}{1-q} \right)^{n-1-j}
		= \frac{q^{j+1}}{(1-q)^n}.
		\]
		This proves the two explicit formulas.
		The corresponding expansions into formal power series, after division by an extra factor $1-q$ (for reasons that will become apparent later), are standard.
	\end{proof}
	
	\begin{lemma}\label{t:Psi_MM_and_FF}
		For any positive integers $m$, $n$ and subsets $J \subseteq [m-1]$ and $K \subseteq [n-1]$,
		\[
		\Psi(M_{m,J} \cdot M_{n,K})
		= (1-q) \sum_{r} \binom{r}{|J|+1} \binom{r}{|K|+1} q^r
		\]
		and
		\[
		\Psi(F_{m,J} \cdot F_{n,K})
		= (1-q) \sum_{r} \binom{r+m-|J|-1}{m} \binom{r+n-|K|-1}{n} q^r.
		\]
	\end{lemma}
	\begin{proof}
		Let $(m_0, \ldots, m_j) = \operatorname{co}_m(J)$  and $(n_0, \ldots, n_k) = \co(K)$ , where $j = |J|$ and $k = |K|$. Then
		\[
		\Psi(M_{m,J} \cdot M_{n,K}) 
		= \sum_{i_0 < \ldots < i_j}
		\sum_{i'_0 < \ldots < i'_k} \Psi(x_{i_0}^{m_0} \cdots x_{i_j}^{m_j} \cdot x_{i'_0}^{n_0} \cdots x_{i'_k}^{n_k}).
		\]
		Since
		\[
		\# \{(i_0, \ldots, i_j, i'_0, \ldots, i'_k) \,:\, i_0 < \ldots  < i_j,\, i'_0 < \ldots < i'_k,\, \max(i_j, i'_k) = r\}
		= \binom{r}{j+1} \binom{r}{k+1} - \binom{r-1}{j+1} \binom{r-1}{k+1},
		\]
		it follows that
		\[
		\Psi(M_{m,J} \cdot M_{n,K}) 
		= \sum_{r} \left[ \binom{r}{j+1} \binom{r}{k+1} - \binom{r-1}{j+1} \binom{r-1}{k+1} \right] q^r
		= (1-q) \sum_{r} \binom{r}{j+1} \binom{r}{k+1} q^r.
		\]
		Consequently,
		\[
		\begin{aligned}
		\Psi(F_{m,J} \cdot F_{n,K}) 
		&= \sum_{J \subseteq J' \subseteq [m-1]} \sum_{K \subseteq K' \subseteq [n-1]} \Psi(M_{m,J'} \cdot M_{n,K'}) \\
		&= (1-q) \sum_{j'=j}^{m-1} \sum_{k'=k}^{n-1} \binom{m-1-j}{j'-j} \binom{n-1-k}{k'-k} \sum_{r} \binom{r}{j'+1} \binom{r}{k'+1} q^r \\
		&= (1-q) \sum_{r} q^r \sum_{j'=j}^{m-1} \binom{m-1-j}{m-1-j'} \binom{r}{j'+1} \sum_{k'=k}^{n-1} \binom{n-1-k}{n-1-k'} \binom{r}{k'+1} \\
		&= (1-q) \sum_{r} \binom{r+m-1-j}{m} \binom{r+n-1-k}{n} q^r.
		\end{aligned}
		\]
	\end{proof}
	
	\begin{theorem}\label{t:des_of_shuffle}
		Let $A$ and $B$ be two disjoint sets of integers,
		with $|A|=m$ and $|B|=n$.
		For each  $u \in \symm_A$ and $v \in \symm_B$,
		the distribution of descent number over all shuffles of
		$u$ and $v$ is given by
		\[
		\sum_{w \in u \shuffle v} q^{\des(w)}
		= (1-q)^{m+n+1} \sum_{r} \binom{r+m-\des(u)}{m} \binom{r+n-\des(v)}{n} q^r.
		\]
	\end{theorem}
	\begin{proof}
		By 
		Theorem~\ref{non-cyclic-product-of-F},
		\[
		F_{m,\Des(u)} \cdot F_{n,\Des(v)}
		= \sum_{w \in u \shuffle v} F_{m+n,\Des(w)}.
		\]
		Apply the mapping $\Psi$ to both sides: 
		on one hand, by Lemma~\ref{t:Psi_MM_and_FF},
		\[
		\Psi(F_{m,\Des(u)} \cdot F_{n,\Des(v)}) 
		= (1-q) \sum_{r} \binom{r+m-\des(u)-1}{m} \binom{r+n-\des(v)-1}{n} q^r.
		\]
		On the other hand, by Lemma~\ref{t:Psi_M_and_F},
		\[
		\Psi \left( \sum_{w \in u \shuffle v} F_{m+n,\Des(w)} \right)
		= \sum_{w \in u \shuffle v} 
		\frac{q^{\des(w)+1}}{(1-q)^{m+n}}.
		\]
		It follows that
		\[
		\begin{aligned}
		\sum_{w \in u \shuffle v} q^{\des(w)}
		&= (1-q)^{m+n+1} \sum_{r} \binom{r+m-\des(u)-1}{m} \binom{r+n-\des(v)-1}{n} q^{r-1} \\
		&= (1-q)^{m+n+1} \sum_{r} \binom{r+m-\des(u)}{m} \binom{r+n-\des(v)}{n} q^r.
		\end{aligned}
		\]
	\end{proof}
	
	We deduce the following result, first proved by Richard Stanley~\cite[Prop.\ 12.6]{Stanley_memoir}. For bijective proofs see~\cite{Goulden, Stadler}.
	
	\begin{corollary}\label{t:num_shuffles_given_des}
		Let $A$ and $B$ be two disjoint sets of integers,
		with $|A|=m$ and $|B|=n$.
		For each  $u \in \symm_A$ and $v \in \symm_B$,
		if $\des(u) = i$ and $\des(v) = j$ then the number of shuffles $w$ 
		in $u \shuffle v$ with $\des(w)=k$ is equal to
		\[
		\binom{m+j-i}{k-i} \binom{n+i-j}{k-j}.
		\]
	\end{corollary}
	\begin{proof}
		By Theorem~\ref{t:des_of_shuffle},
		the number of shuffles of $u$ and $v$ with descent number $k$ is equal to the coefficient of $q^k$ in the product
		\[
		(1-q)^{m+n+1} \sum_{r} \binom{r+m-i}{m} \binom{r+n-j}{n} q^r.
		\]
		We therefore want to show that
		\[
		(1-q)^{m+n+1} \sum_{r} \binom{r+m-i}{m} \binom{r+n-j}{n} q^r
		= \sum_t \binom{m+j-i}{t-i} \binom{n+i-j}{t-j} q^t.
		\]
		
		We shall use the triple-binomial identity~\cite[(5.28)]{GKP}
		\begin{equation}\label{e.pfaff}
			\sum_{\ell} \binom{m-x+y}{\ell} \binom{n+x-y}{n-\ell} \binom{x+\ell}{m+n}
			= \binom{x}{m} \binom{y}{n},
		\end{equation}
		valid for nonnegative integers $m$ and $n$ and arbitrary commuting indeterminates $x$ and $y$; summation is over the admissible integer values of $\ell$, say $0 \le \ell \le n$. 
		It is equivalent to the Pfaff-Saalsch\"{u}tz hypergeometric identity~\cite[(5.97)]{GKP}; we have renamed some of the variables to avoid confusion with our current notation.
		
		In \eqref{e.pfaff} substitute $x := m+r-i$, $y := n+r-j$ and $\ell := n+i-t$ to get
		\begin{equation}\label{e.pfaff1}
			\sum_t \binom{n+i-j}{n+i-t} \binom{m+j-i}{t-i} \binom{m+n+r-t}{m+n}
			= \binom{m+r-i}{m} \binom{n+r-j}{n}.
		\end{equation}
		It follows that
		\begin{equation}\label{e.pfaff2}
			\sum_r \sum_{t\le r}  \binom{m+j-i}{t-i} \binom{n+i-j}{t-j} \binom{m+n+r-t}{m+n} q^r
			= \sum_r \binom{r+m-i}{m} \binom{r+n-j}{n} q^r.
		\end{equation}
		Since
		\[
		(1-q)^{-(m+n+1)} = \sum_{s \ge 0} \binom{m+n+s}{m+n} q^s,
		\]
		the LHS of \eqref{e.pfaff2} is equal to
		\[
		(1-q)^{-(m+n+1)} \sum_t \binom{m+j-i}{t-i} \binom{n+i-j}{t-j} q^t,
		\]
		completing the proof.
		
	\end{proof}
	
	Now we turn to the cyclic analogues.
	
	\begin{lemma}\label{t:Psi_cyc_M_and_F}
		For any positive integer $n$ and subset $J \subseteq [n]$,
		\[
		\Psi(M^{\cyc}_{n,J}) 
		= \frac{|J| q^{|J|}}{(1-q)^{|J|}}
		= (1-q) \sum_r |J| \binom{r}{|J|} q^r
		\]
		and
		\[
		\Psi(F^{\cyc}_{n,J}) 
		= \frac{|J| q^{|J|} + (n - |J|) q^{|J|+1}}{(1-q)^{n}}
		= (1-q) \sum_r \binom{r+n-|J|-1}{n-1} r q^r.
		\]
	\end{lemma}
	\begin{proof}
		By Lemma~\ref{t:Mcyc_from_M} and Lemma~\ref{t:Psi_M_and_F},
		\[
		\Psi(M^{\cyc}_{n,J}) = \sum_{j \in J} \Psi(M_{n, (J - j) \cap [n-1]})
		= \frac{|J| q^{|J|}}{(1-q)^{|J|}}
		= (1-q) \sum_r |J| \binom{r}{|J|} q^r.
		\]
		By Lemma~\ref{t:Fcyc_from_F} and Lemma~\ref{t:Psi_M_and_F},
		\[
		\Psi(F^{\cyc}_{n,J}) = \sum_{i \in [n]} \Psi(F_{n, (J - i) \cap [n-1]})
		= \frac{|J| q^{|J|} + (n - |J|) q^{|J|+1}}{(1-q)^{n}}
		= (1-q) \sum_r c_r(n,|J|) q^r,
		\]
		where
		\[
		\begin{aligned}
		c_r(n,j) 
		&= j \binom{r+n-j}{n} + (n-j) \binom{r+n-j-1}{n} \\
		&= \binom{r+n-j-1}{n-1} \cdot \frac{j(r+n-j) + (n-j)(r-j)}{n} \\
		&= \binom{r+n-j-1}{n-1} \cdot r.
		\end{aligned}
		\]
	\end{proof}
	
	\begin{lemma}\label{t:Psi_cyc_MM_and_FF}
		For any positive integers $m$, $n$ and subsets $J \subseteq [m]$ and $K \subseteq [n]$,
		\[
		\Psi(M^{\cyc}_{m,J} \cdot M^{\cyc}_{n,K}) 
		= (1-q) \sum_{r} |J| \binom{r}{|J|} |K| \binom{r}{|K|} q^r
		\]
		and
		\[
		\Psi(F^{\cyc}_{m,J} \cdot F^{\cyc}_{n,K}) 
		= (1-q) \sum_{r} \binom{r+m-|J|-1}{m-1} \binom{r+n-|K|-1}{n-1} r^2 q^r.
		\]
	\end{lemma}
	\begin{proof}
		By Lemma~\ref{t:Mcyc_from_M} and Lemma~\ref{t:Psi_MM_and_FF},
		\[
		\begin{aligned}
		\Psi(M^{\cyc}_{m,J} \cdot M^{\cyc}_{n,K}) 
		&= \sum_{j \in J} \sum_{k \in K} \Psi(M_{m, (J - j) \cap [m-1]} \cdot M_{n, (K - k) \cap [n-1]}) \\
		&= |J| \cdot |K| \cdot (1-q) \sum_{r} \binom{r}{|J|} \binom{r}{|K|} q^r.
		\end{aligned}
		\]
		By Lemma~\ref{t:Fcyc_from_F} and Lemma~\ref{t:Psi_MM_and_FF},
		\[
		\begin{aligned}
		\Psi(F^{\cyc}_{m,J} \cdot F^{\cyc}_{n,K}) 
		&= \sum_{j \in [m]} \sum_{k \in [n]} \Psi(F_{m, (J - j) \cap [m-1]} \cdot F_{n, (K - k) \cap [n-1]}) \\
		&= (1-q) \sum_{r} c_r(m,|J|) \cdot c_r(n,|K|) \cdot q^r
		\end{aligned}
		\]
		where, as in the proof of Lemma~\ref{t:Psi_cyc_M_and_F},
		\[
		\begin{aligned}
		c_r(m,j) 
		&= j \binom{r+m-j}{m} + (m-j) \binom{r+m-j-1}{m} \\
		&= \binom{r+m-j-1}{m-1} \cdot \frac{j(r+m-j) + (m-j)(r-j)}{m} \\
		&= \binom{r+m-j-1}{m-1} \cdot r.
		\end{aligned}
		\]
	\end{proof}
	
	\begin{theorem}\label{t:cdes_of_cyclic_shuffle}
		Let $A$ and $B$ be two disjoint sets of integers,
		with $|A| = m$ and $|B| = n$.
		For each  $u \in \symm_A$ and $v \in \symm_B$,
		the distribution of cyclic descent number over all cyclic shuffles of
		$[\DAG u]$ and $[\DAG v]$ is given by
		\[
		\sum_{[\DAG{w}] \in [\DAG{u}] \shuffle_{\cyc} [\DAG{v}]} q^{\cdes(w)}
		= (1-q)^{m+n} \sum_r \binom{r+m-\cdes(u)-1}{m-1} \binom{r+n-\cdes(v)-1}{n-1} r q^r.
		\]
	\end{theorem}
	\begin{proof}
		By 
		Theorem~\ref{product-of-F-corollary}, 
		\[
		F^\cyc_{m,\cDes(u)} \cdot F^\cyc_{n,\cDes(v)}
		= \sum_{[\DAG{w}] \in [\DAG{u}] \shuffle_{\cyc} [\DAG{v}]} F^\cyc_{m+n,\cDes(w)}.
		\]
		Apply the mapping $\Psi$ to both sides: 
		on one hand, by Lemma~\ref{t:Psi_cyc_MM_and_FF},
		\[
		\Psi(F^{\cyc}_{m,\cDes(u)} \cdot F^{\cyc}_{n,\cDes(v)}) 
		= (1-q) \sum_{r} \binom{r+m-\cdes(u)-1}{m-1} \binom{r+n-\cdes(v)-1}{n-1} r^2 q^r.
		\]
		On the other hand, by Lemma~\ref{t:Psi_cyc_M_and_F},
		\[
		\Psi \left( \sum_{[\DAG{w}] \in [\DAG{u}] \shuffle_{\cyc} [\DAG{v}]} F^\cyc_{m+n,\cDes(w)} \right)
		= (1-q) \sum_{[\DAG{w}] \in [\DAG{u}] \shuffle_{\cyc} [\DAG{v}]} \sum_r \binom{r+m+n-\cdes(w)-1}{m+n-1} r q^r.
		\]
		It follows that
		\[
		\sum_{[\DAG{w}] \in [\DAG{u}] \shuffle_{\cyc} [\DAG{v}]} \binom{r+m+n-\cdes(w)-1}{m+n-1}
		= \binom{r+m-\cdes(u)-1}{m-1} \binom{r+n-\cdes(v)-1}{n-1} r
		\qquad (\forall r > 0).
		\]
		This also holds for $r=0$, since the cyclic descent number of a permutation is strictly positive, so that both sides are equal to zero. Using
		\[
		\frac{q^c}{(1-q)^{m+n}} 
		= \sum_r \binom{r+m+n-c-1}{m+n-1} q^r
		\qquad (\forall c > 0),
		\]
		this can be written, equivalently, as
		\[
		\sum_{[\DAG{w}] \in [\DAG{u}] \shuffle_{\cyc} [\DAG{v}]} q^{\cdes(w)}
		= (1-q)^{m+n} \sum_r \binom{r+m-\cdes(u)-1}{m-1} \binom{r+n-\cdes(v)-1}{n-1} r q^r.
		\]
	\end{proof}
	
	\begin{corollary}\label{t:num_shuffles_given_cdes}
		Let $A$ and $B$ be two disjoint sets of integers,
		with $|A| = m$ and $|B| = n$.
		For each  $u \in \symm_A$ and $v \in \symm_B$,
		if $\cdes(u) = i$ and $\cdes(v) = j$ then the number of cyclic shuffles 
		$[\DAG w]$ in $[\DAG{u}] \shuffle_{\cyc} [\DAG{v}]$ with $\cdes(w)=k$ is equal to
		\[
		\begin{aligned}
		a(m,n,i,j,k) 
		&:= k \binom{m+j-i-1}{k-i} \binom{n+i-j-1}{k-j}
		+ (m+n-k) \binom{m+j-i-1}{k-i-1} \binom{n+i-j-1}{k-j-1} \\
		&\;= \frac{k(m-i)(n-j) + (m+n-k)ij}{(m+j-i)(n+i-j)} \binom{m+j-i}{k-i} \binom{n+i-j}{k-j}. 
		\end{aligned}
		\]
	\end{corollary}
	\begin{proof}
		Similar to the proof of Corollary~\ref{t:num_shuffles_given_des} above.
		By Theorem~\ref{t:cdes_of_cyclic_shuffle}, the number of cyclic shuffles of $[\DAG u]$ and $[\DAG v]$ with cyclic descent number $k$ is equal to the coefficient of $q^k$ in the product
		\[
		(1-q)^{m+n} \sum_r \binom{r+m-i-1}{m-1} \binom{r+n-j-1}{n-1} r q^r.
		\]
		We therefore want to show that
		\[
		(1-q)^{m+n} \sum_r \binom{r+m-i-1}{m-1} \binom{r+n-j-1}{n-1} r q^r
		= \sum_t a(m,n,i,j,t) q^t.
		\]
		
		Using identity~\eqref{e.pfaff1} 
		with $m-1$ and $n-1$ instead of $m$ and $n$, it follows that
		\begin{equation}\label{e.pfaff2c}
			\begin{aligned}
				&\, \sum_r \binom{r+m-i-1}{m-1} \binom{r+n-j-1}{n-1} r q^r \\
				= &\, \sum_r \sum_{t\le r}  \binom{m+j-i-1}{t-i} \binom{n+i-j-1}{t-j} \binom{r-t+m+n-2}{m+n-2} r q^r.
			\end{aligned}
		\end{equation}
		Since
		\[
		(1-q)^{-(m+n)} = \sum_{s \ge 0} \binom{s+m+n-1}{m+n-1} q^s
		\]
		and
		\[
		r \binom{r-t+m+n-2}{m+n-2} 
		=  t \binom{r-t+m+n-1}{m+n-1} + (m+n-t-1) \binom{r-t-1+m+n-1}{m+n-1},
		\]
		it follows that the RHS of~\eqref{e.pfaff2c} is equal to $(1-q)^{-(m+n)}$ times
		\[
		\begin{aligned}
		& 
		\sum_t \binom{m+j-i-1}{t-i} \binom{n+i-j-1}{t-j} t q^t
		+ \sum_t \binom{m+j-i-1}{t-i} \binom{n+i-j-1}{t-j} (m+n-t-1) q^{t+1} \\
		=&
		\sum_t \left[ \binom{m+j-i-1}{t-i} \binom{n+i-j-1}{t-j} t
		+ \binom{m+j-i-1}{t-i-1} \binom{n+i-j-1}{t-j-1} (m+n-t) \right] q^t
		\end{aligned}
		\]
		so that
		\[
		(1-q)^{m+n} \sum_r \binom{r+m-i-1}{m-1} \binom{r+n-j-1}{n-1} r q^r
		= \sum_t a(m,n,i,j,t) q^t,
		\]
		as claimed. 
	\end{proof}
	
	\subsection{Distribution of cyclic descents over SYT}
	
	
	Let us first restate Corollary~\ref{t:Schur_in_hFcyc} in an equivalent form.
	
	\begin{lemma}\label{cor:QS1}
		For any skew shape $\lambda/\mu$ of size $n$, which is not a connected ribbon, 
		for any cyclic extension $(\cDes,p)$ of $\Des$ on $\SYT(\lambda/\mu)$,
		and for any subset $\varnothing \subsetneq J \subsetneq [n]$,
		the fiber size 
		\[
		\#\{T\in \SYT(\lambda/\mu):\ \cDes(T)=J\}
		\]
		is equal to 
		the coefficient of $\hF^{\cyc}_{n,[J]}$ in the expansion of $s_{\lambda/\mu}$ in terms of normalized fundamental cyclic quasi-symmetric functions. 
	\end{lemma}
	
	It follows that identities involving Schur functions may be interpreted as statements about the distribution of $\cDes$ over standard Young tableaux.
	Several examples will be given below.
	
	\bigskip
	
	
	
	Write $\lambda/\mu = \nu^1 \oplus \cdots \oplus \nu^t$ to indicate that 
	$\nu^1, \dots,\nu^t$ are the connected components of the skew shape $\lambda/\mu$, ordered from southwest to northeast. 
	It is well known that 
	\begin{equation}\label{eq:prod}
		s_{\nu^1\oplus \cdots \oplus  \nu^t}=s_{\nu^1}\cdots s_{\nu^t}.
	\end{equation}
	
	\medskip
	
	For a subset $J\subseteq [n]$ denote ${\bf x}^J:=\prod\limits_{i\in J}x_i$.
	
	\begin{proposition}\label{prop:disconnected_skew1}
		Let 
		$\lambda \vdash m$ and $\mu \vdash n$ be non-hook partitions. 
		If $A_{\lambda} \subseteq \symm_{[m]}$
		and $A_{\mu} \subseteq \symm_{[m+1,m+n]}$ are sets of permutations satisfying\footnote{Such sets $A_\lambda, A_\mu$ are known to exist both from \cite[Theorem 4.4]{ARR_CDes} and from more recent work of B. Huang \cite{Huang}.},
		\[
		\sum_{T \in \SYT(\lambda)} {\bf x}^{\cDes(T)}
		= \sum_{\pi \in A_{\lambda}} {\bf x}^{\cDes(\pi)} 
		\qquad \text{\rm and} \qquad 
		\sum_{T \in \SYT(\mu)} {\bf x}^{\cDes(T)} 
		= \sum_{\pi \in A_{\mu}} {\bf x}^{\cDes(\pi)}
		\]
		then one has 
		\[
		\sum_{T \in \SYT(\lambda \oplus \mu)} {\bf x}^{\cDes(T)} 
		= \frac{1}{mn} \sum_{\substack{\sigma \in A_\lambda \\ \tau \in A_\mu}} 
		\sum_{\substack{w \in \symm_{m+n} \\ [\DAG w] \in [\DAG \sigma] \shuffle_{\cyc} [\DAG \tau]}} {\bf x}^{\cDes(w)}.
		\]
	\end{proposition}
	
	\begin{proof} 
		First, by Theorem~\ref{conj:QS1}, 
		\[
		s_{\lambda \oplus \mu} = \frac{1}{m+n} \sum_{T \in \SYT(\lambda \oplus \mu)} F^\cyc_{m+n, \cDes(T)}.
		\]
		On the other hand, 
		by Equation~\eqref{eq:prod} combined  with Theorem~\ref{conj:QS1}, Theorem~\ref{product-of-F-corollary}, and our assumptions regarding the equidistribution of $\cDes$ on corresponding sets of permutations and tableaux,
		\[
		\begin{aligned}
		s_{\lambda\oplus \mu} = s_\lambda s_\mu 
		&= \frac{1}{m} \sum_{T \in \SYT(\lambda)} F^\cyc_{m, \cDes(T)} \cdot \frac{1}{n} \sum_{T \in \SYT(\mu)} F^\cyc_{n, \cDes(T)} \\
		&= \frac{1}{mn} \sum_{\substack{\sigma \in A_\lambda \\ \tau \in A_\mu}} F^\cyc_{m, \cDes(\sigma)} F^\cyc_{n, \cDes(\tau)} \\
		&= \frac{1}{mn} \sum_{\substack{\sigma \in A_\lambda \\ \tau \in A_\mu}} \sum_{[\DAG w] \in [\DAG \sigma] \shuffle_{\cyc} [\DAG \tau]} F^\cyc_{m+n, \cDes(w)} \\
		&= \frac{1}{mn(m+n)} \sum_{\substack{\sigma \in A_\lambda \\ \tau \in A_\mu}} \sum_{\substack{w \in \symm_{m+n} \\ [\DAG w] \in [\DAG \sigma] \shuffle_{\cyc} [\DAG \tau]}} F^\cyc_{m+n, \cDes(w)}.
		\end{aligned}
		\]
		One concludes that
		\[
		\sum_{T \in \SYT(\lambda \oplus \mu)} F^\cyc_{m+n, \cDes(T)} 
		= \frac{1}{mn} \sum_{\substack{\sigma \in A_\lambda \\ \tau \in A_\mu}} \sum_{\substack{w \in \symm_{m+n} \\ [\DAG w]\in [\DAG \sigma] \shuffle_{\cyc} [\DAG \tau]}} F^\cyc_{m+n, \cDes(w)}.
		\]
		Now, fiber sizes of the $\cDes$ map on both 
		$\SYT(\lambda\oplus \mu)$ and $\{w\in \symm_{n+m}:\ [\DAG w]\in [\DAG \sigma] \shuffle_{\cyc} [\DAG \tau]\}$ are invariant under cyclic rotation.
		Comparing coefficients of $F^\cyc_{m+n, J}$ on both sides of the above equation thus completes the proof.
	\end{proof}

	The next application of Lemma~\ref{cor:QS1} relates the distribution of $\cDes$ on permutations  
	to its distribution on tableaux of various straight shapes $\lambda$. 
	
	
	For every positive integer $n$ define the corresponding {\em multivariate cyclic Eulerian polynomial} by
	\[
	\symm_n^\cDes(\ttt)
	:= \symm_n^{\cDes}(t_1,\ldots,t_n)
	:= \sum_{\pi \in \symm_n} \prod_{i \in \cDes(\pi)} t_i. 
	\]
	
	For every skew shape $\lambda/\mu$ which is not a connected ribbon denote
	\[
	\symm_{\lambda/\mu}^\cDes(\ttt) 
	:= \symm_{\lambda/\mu}^\cDes(t_1,\ldots,t_n)
	:= \sum_{T \in \SYT(\lambda/\mu)} \prod_{i \in \cDes(T)} t_i.
	\]
	
	Let $f^\lambda := \#\SYT(\lambda)$.
	
	\smallskip
	
	\begin{theorem}\label{conj22}
		{\rm (\cite[Theorem 1.2]{ARR_CDes})}\\
		For every positive integer $n$
		\[
		\symm_n^\cDes(\ttt)
		= \sum_{\substack{\textnormal{non-hook}\\ \lambda \vdash n}}
		f^\lambda
		\symm_{\lambda}^\cDes(\ttt) 
		+ \sum_{k=1}^{n-1} \binom{n-2}{k-1} 
		\symm_{(1^k) \oplus (n-k)}^\cDes(\ttt),
		\]
		where the last summation is over skew shapes $(1^k) \oplus (n-k)$, $1\le k\le n-1$, consisting of one column of size $k$ and one row of size $n-k$.
	\end{theorem}
	
	\begin{proof} 
		First, by 
		Pieri's rule
		\begin{equation}\label{eq:nearhook}
			s_{(1^k) \oplus (n-k)}
			= s_{(1^k)} s_{(n-k)}
			= s_{(n-k,1^{k})} + s_{(n-k+1,1^{k-1})}.
		\end{equation}
		Hence
		\[
		\sum_{k=1}^{n-1} \binom{n-2}{k-1} s_{(1^k) \oplus (n-k)}
		= \sum_{k=0}^{n-1} \binom{n-1}{k} s_{(n-k,1^{k})}
		= \sum_{k=0}^{n-1} f^{(n-k,1^{k})} s_{(n-k,1^{k})}.
		\]
		Combining this with Equation~\eqref{eq:prod} and iterations of Pieri's rule we get
		\[
		s_{(1)^{\oplus n}}
		= s_1^n
		= \sum_{\lambda \vdash n} f^\lambda s_\lambda
		= \sum_{\substack{\textnormal{non-hook}\\ \lambda \vdash n}} f^\lambda s_\lambda 
		+ \sum_{k=1}^{n-1} \binom{n-2}{k-1} s_{(1^k) \oplus (n-k)}.
		\]
		The bijection $f: \symm_n \to \SYT(1^{\oplus n})$, which sends a permutation $w$ to the standard Young tableau whose entries are
		$w^{-1}(1),\ldots, w^{-1}(n)$ (read from southwest to northeast),
		preserves descent sets and therefore also the distribution of cyclic descent sets.
		Lemma~\ref{cor:QS1} completes the proof.
	\end{proof}
	
	\medskip
	
	In a similar fashion one can simplify the proofs of Theorem 5.1 and Corollary 7.1 
	in~\cite{ARR_CDes}.
	
	
	%
	
	\bigskip
	
	Here is another, more specific, application.
	For $2 \le k \le n-2$ consider two shapes obtained by adding a cell to the hook $(n-k,1^{k-1}) \vdash n-1$:
	\begin{enumerate}
		\item
		Add a cell at the inner corner, to get $(n-k,2,1^{k-2})$.
		\item
		Add a disconnected cell at the northeast corner, to get $(n-k,1^{k-1}) \oplus (1)$.
	\end{enumerate}
	
	\begin{example} 
		For $n=6$ and $k=3$, the two shapes obtained from the hook $(3,1,1)$ are
		\[
		(3,2,1) 
		= \young(\hfill\hfill\hfill,\hfill\hfill,\hfill) 
		\qquad \text{and}\qquad
		(3,1,1) \oplus (1)
		= \young(:::\hfill,\hfill\hfill\hfill,\hfill,\hfill) \qquad .
		\]
	\end{example}
	
	The following proposition shows that the fiber sizes of cyclic descent maps on these two shapes differ by at most one.
	This explains the slight difference between Theorem 3.11(b) and Theorem 3.11(c) (see also Proposition 3.17 and Proposition 3.18) in~\cite{AER}.
	
	\begin{proposition}\label{prop:near-hooks}
		For every $2 \le k \le n-2$ and $\varnothing \subsetneq J \subsetneq [n]$, if $|J| = k$ then
		\[
		\#\{T\in \SYT((n-k,1^{k-1})\oplus (1)):\ \cDes(T)=J\}
		- \#\{T\in \SYT(n-k,2,1^{k-2}):\ \cDes(T)=J\}
		= 1.
		\]
		If $|J| \ne k$ then 
		this difference is $0$.
	\end{proposition}
	
	\begin{proof}
		By 
		Pieri's rule,
		\[
		s_{(n-k,1^{k-1}) \oplus (1)} 
		= s_{(n-k,1^{k-1})} s_{(1)}
		= s_{(n-k+1,1^{k-1})} + s_{(n-k,2,1^{k-2})} + s_{(n-k,1^{k})}.
		\]
		Combining this with Equation~\eqref{eq:nearhook} one obtains
		\[
		s_{(n-k,1^{k-1}) \oplus (1)} 
		= s_{(n-k,2,1^{k-2})} + s_{(1^{k}) \oplus (n-k)}.
		\]
		By Lemma~\ref{cor:QS1} this identity is equivalent to
		\[
		\sum_{T \in \SYT((n-k,1^{k-1}) \oplus (1))} {\bf x}^{\cDes(T)}
		= \sum_{T \in \SYT((n-k,2,1^{k-2}))} {\bf x}^{\cDes(T)}
		+ \sum_{T \in \SYT((1^{k}) \oplus (n-k))} {\bf x}^{\cDes(T)}.
		\]
		Equation~\eqref{eq:oplus}, describing the second summand on the RHS, completes the proof.
	\end{proof}

	\section{The internal coproduct and the cyclic descent module}\label{sec:co}
	
	The ring of quasi-symmetric functions admits two natural coproducts:
	the {\em inner} (or {\em internal}) coproduct, whose dual is anti-isomorphic to the product in Solomon's descent algebra~\cite{Gessel}, 
	and the {\em outer} (or {\em graded}) coproduct, whose dual is the product in the Hopf algebra of non-commutative symmetric functions~\cite{Gelfand, MalvenutoReutenauer}.
	The current section studies a cyclic analogue of the inner coproduct; 
	we leave the search for an analogue of the outer coproduct to the interested reader.
	
	
	Let us start with a short review of the internal coproduct on $\QSym$, following~\cite{Gessel}.
	Consider two countable sets of variables,
	$X = \{x_1 < x_2 < \ldots\}$ and $Y = \{y_1 < y_2 < \ldots\}$,
	labeled by positive integers and (totally) ordered accordingly. 
	Let $XY := \{x_i y_j \,:\, i, j \ge 1\}$, ordered lexicographically:
	\[
	x_{i_1} y_{j_1} < x_{i_2} y_{j_2} \iff
	i_1 < i_2 \text{ or } (i_1 = i_2 \text{ and } j_1 < j_2).
	\]
	It turns out \cite[Theorem 11]{Gessel} that fundamental quasi-symmetric functions in the variables $XY$ can be expressed in terms of corresponding functions of $X$ and $Y$ separately:
	\begin{equation}\label{eq:F_to_F_times_F}
		F_{n,\Des(\pi)}(XY) 
		= \sum_{\substack{\sigma_1, \sigma_2 \in \symm_n \\ \sigma_2 \sigma_1 = \pi}} F_{n,\Des(\sigma_1)}(X) \cdot F_{n,\Des(\sigma_2)}(Y)
		\qquad (\forall \pi \in \symm_n).
	\end{equation}
	It follows that, for any three subsets $I,J,K \subseteq [n-1]$, the nonnegative integer
	\[
	a_{K}^{IJ} := \#\{(\sigma_1, \sigma_2) \in \symm_n \times \symm_n \,:\, \Des(\sigma_1) = I,\, \Des(\sigma_2) = J,\, \sigma_2 \sigma_1 = \pi\}
	\]
	indeed depends only on ($I$, $J$ and) $K$ but not on the choice of $\pi$, as long as $\Des(\pi) = K$.
	We can therefore write
	\[
	F_{n,K}(XY) 
	= \sum_{I, J \subseteq [n-1]}
	a_{K}^{IJ} F_{n,I}(X) \cdot F_{n,J}(Y),
	\]
	and define a coproduct $\Delta_n: \QSym_n \to \QSym_n \otimes \QSym_n$ by
	\[
	\Delta_n(F_{n,K})
	:= \sum_{I, J \subseteq [n-1]}
	a_{K}^{IJ} F_{n,I} \otimes F_{n,J}
	\qquad (\forall K \subseteq [n-1]).
	\]
	This turns $\QSym_n$ into a coalgebra~\cite{Gessel}. 
	
	There is also a dual structure.
	Consider the following elements of the group ring $\ZZ[\symm_n]$:
	\[
	D_I 
	:= \sum_{\substack{\pi \in \symm_n \\ \Des(\pi) = I}} \pi
	\qquad (I \subseteq [n-1]).
	\]
	The discussion above implies that
	\[
	D_J \cdot D_I = \sum_{K \subseteq [n-1]} a_{K}^{IJ} D_K
	\qquad (\forall \,I, J \subseteq [n-1]),
	\]
	where multiplication is that of the group ring.
	It follows that the additive free abelian group
	\[
	\mathfrak{D}_n := \spn_{\,\ZZ} \{D_I \,:\, I \subseteq [n-1]\}
	\]
	is actually a subring of $\ZZ[\symm_n]$, known as {\em Solomon's descent algebra}
	\cite{Solomon}. 
	
	\bigskip
	
	The ring $\cQSym$ of cyclic quasi-symmetric functions is a subring of $\QSym$. What is its status with respect to the internal coproduct? Is there a dual structure?
	
	
	\begin{theorem}\label{t:cQSym_right_comodule}
		$\cQSym_n$ and $\cQSym_n^-$ are right comodules of $\QSym_n$ with respect to the internal coproduct:
		\[
		\Delta_n(\cQSym_n) \subseteq \cQSym_n \otimes \QSym_n \qquad (\forall n \ge 0)
		\]
		and
		\[
		\Delta_n(\cQSym_n^-) \subseteq \cQSym_n^- \otimes \QSym_n \qquad (\forall n \ge 0).
		\]
		Explicitly, in $\cQSym_n^-$ (for $n > 1$):
		\[
		\Delta_n(\hF^{\cyc}_{n,B})
		= \sum_{\substack{A \in c2_{0,n}^{[n]} \\ J \subseteq [n-1]}}
		\frac{d_A}{d_B} \tilde{a}_{B}^{AJ} \hF^{\cyc}_{n,A} \otimes F_{n,J}
		\qquad (\forall \,B \in c2_{0,n}^{[n]}),
		\]
		where the number 
		\[
		\tilde{a}_{B}^{AJ} := \#\{(\sigma_1, \sigma_2) \in \symm_n \times \symm_n \,:\, \cDes(\sigma_1) \in A,\, \Des(\sigma_2) = J,\, \sigma_2 \sigma_1 = \pi\}
		\]
		depends on $A, B \in c2_{0,n}^{[n]}$ and $J \subseteq [n-1]$ but not on $\pi$, as long as $\cDes(\pi) \in B$.
		The structure constants $\frac{d_A}{d_B} \tilde{a}_{B}^{AJ}$ are nonnegative integers.
		In addition, for $n \ge 1$:
		\[
		\Delta_n(\hF^{\cyc}_{n,\varnothing}) 
		= \Delta_n(h_n) 
		= \sum_{\lambda \vdash n} s_{\lambda} \otimes s_{\lambda}
		\]
		and
		\[
		\Delta_n(\hF^{\cyc}_{n,[n]}) 
		= \Delta_n(e_n) 
		= \sum_{\lambda \vdash n} s_{\lambda} \otimes s_{\lambda'}
		\]
		belong to $\Sym_n \otimes \Sym_n \subseteq \cQSym_n \otimes \QSym_n$.
	\end{theorem}
	
	\begin{proof}
		First, consider a permutation $\pi \in \symm_n$. By Proposition~\ref{t:Fcyc_from_F},
		\[
		F^{\cyc}_{n,\cDes(\pi)} 
		= \sum_{i \in [n]} F_{n, (\cDes(\pi) - i) \cap [n-1]}.
		\]
		Denoting by $c \in \symm_n$ the permutation mapping
		$i \mapsto i+1$ ($1 \le i \le n-1$) and $n \mapsto 1$, clearly
		\[
		\cDes(\pi) - i = \cDes(\pi c^i) \qquad (\forall i)
		\]
		and therefore
		\[
		F^{\cyc}_{n,\cDes(\pi)} 
		= \sum_{i \in [n]} F_{n, \cDes(\pi c^i) \cap [n-1]}
		= \sum_{i \in [n]} F_{n, \Des(\pi c^i)}.
		\]
		Using \eqref{eq:F_to_F_times_F}, it follows that
		\[
		\begin{aligned}
		F^{\cyc}_{n,\cDes(\pi)}(XY)
		&= \sum_{i \in [n]} F_{n, \Des(\pi c^i)}(XY) \\
		&= \sum_{i \in [n]} \sum_{\substack{\sigma_1, \sigma_2 \in \symm_n \\ \sigma_2 \sigma_1 = \pi}} F_{n,\Des(\sigma_1 c^i)}(X) \cdot F_{n,\Des(\sigma_2)}(Y) \\
		&= \sum_{\substack{\sigma_1, \sigma_2 \in \symm_n \\ \sigma_2 \sigma_1 = \pi}} F^{\cyc}_{n,\cDes(\sigma_1)}(X) \cdot F_{n,\Des(\sigma_2)}(Y),
		\end{aligned}
		\]
		and therefore
		\[
		\Delta_n(F^{\cyc}_{n,\cDes(\pi)})
		= \sum_{\substack{\sigma_1, \sigma_2 \in \symm_n \\ \sigma_2 \sigma_1 = \pi}}
		F^{\cyc}_{n,\cDes(\sigma_1)} \otimes F_{n,\Des(\sigma_2)} \in \cQSym_n \otimes \QSym_n
		\qquad (\forall \pi \in \symm_n).
		\]
		Denoting $A := [\cDes(\sigma_1)]$, $J := \Des(\sigma_2)$ and $B := [\cDes(\pi)]$, this can be written as
		\[
		\Delta_n(F^{\cyc}_{n,B})
		= \sum_{\substack{A \in c2_{0,n}^{[n]} \\ J \subseteq [n-1]}}
		\tilde{a}_{B}^{AJ} F^{\cyc}_{n,A} \otimes F_{n,J}
		\qquad (\forall \,B \in c2_{0,n}^{[n]})
		\]
		and the normalized version follows.
		The structure constants $\frac{d_A}{d_B} \tilde{a}_{B}^{AJ}$ are clearly nonnegative rational numbers. We can prove their integrality by a general indirect argument, as follows:
		$\hF^{\cyc}_{n,B} \in \cQSym_n \subseteq \QSym_n$, and therefore $\Delta_n(\hF^{\cyc}_{n,B}) \in \QSym_n \otimes \QSym_n$ is a linear combination, with integer coefficients, of the basis elements $F_{n,I} \otimes F_{n,J}$.
		Expanded in $\ZZ[[X]]_n \otimes \ZZ[[Y]]_n$, it is a linear combination of tensor products of monomials, again with integer coefficients.
		On the other hand, the above computation shows that it belongs to $\QQ \otimes \cQSym_n \otimes \QSym_n$, and thus is a linear combination of basis elements $\hF^{\cyc}_{n,A} \otimes F_{n,J}$, or equivalently basis elements $\hM^{\cyc}_{n,A} \otimes M_{n,J}$, with rational coefficients. 
		Each monomial in $\ZZ[[X]]_n \otimes \ZZ[[Y]]_n$ appears in exactly one of the latter basis elements, with coefficient $1$, and thus the coefficients in this basis (as well as in the other basis) are actually integers. 
		
		Any subset of $[n]$, other than $\varnothing$ and $[n]$, has the form $\cDes(\pi)$ for some $\pi \in \symm_n$;
		it remains to check $\Delta_n(\hF^{\cyc}_{n,\varnothing})$ and $\Delta_n(\hF^{\cyc}_{n,[n]})$. 
		Indeed,
		\[
		\begin{aligned}
		\Delta_n(\hF^{\cyc}_{n,\varnothing}) 
		&= \Delta_n(F_{n,\varnothing})
		= \Delta_n(F_{n,\Des(id)})
		= \sum_{\substack{\sigma_1, \sigma_2 \in \symm_n \\ \sigma_2 \sigma_1 = id}}
		F_{n,\Des(\sigma_1)} \otimes F_{n,\Des(\sigma_2)} \\
		&= \sum_{\sigma \in \symm_n}
		F_{n,\Des(\sigma)} \otimes F_{n,\Des(\sigma^{-1})}
		\end{aligned}
		\]
		and, denoting by $w_0$ the involution (longest element) in $\symm_n$ mapping $i \mapsto n+1-i$ $(\forall i \in [n])$,
		\[
		\begin{aligned}
		\Delta_n(\hF^{\cyc}_{n,[n]}) 
		&= \Delta_n(F_{n,[n-1]})
		= \Delta_n(F_{n,\Des(w_0)})
		= \sum_{\substack{\sigma_1, \sigma_2 \in \symm_n \\ \sigma_2 \sigma_1 = w_0}}
		F_{n,\Des(\sigma_1)} \otimes F_{n,\Des(\sigma_2)} \\
		&= \sum_{\sigma \in \symm_n}
		F_{n,\Des(\sigma w_0)} \otimes F_{n,\Des(\sigma^{-1})}.
		\end{aligned}
		\]
		Recall now \cite[Theorem 7.23.2]{EC2}, which can be stated as the pair of identities
		\[
		\sum_{\lambda \vdash n} s_{\lambda} \otimes s_{\lambda}
		= \sum_{\sigma \in \symm_n} F_{n,\Des(\sigma)} \otimes F_{n,\Des(\sigma^{-1})}
		\]
		and
		\[
		\sum_{\lambda \vdash n} s_{\lambda} \otimes s_{\lambda'}
		= \sum_{\sigma \in \symm_n} F_{n,\Des(\sigma w_0)} \otimes F_{n,\Des(\sigma^{-1})}.
		\]
		Using them leads to the claimed expressions
		\[
		\Delta_n(\hF^{\cyc}_{n,\varnothing}) 
		= \sum_{\lambda \vdash n} s_{\lambda} \otimes s_{\lambda}
		\in \Sym_n \otimes \Sym_n
		\]
		and
		\[
		\Delta_n(\hF^{\cyc}_{n,[n]}) 
		= \sum_{\lambda \vdash n} s_{\lambda} \otimes s_{\lambda'}
		\in \Sym_n \otimes \Sym_n.
		\]
		
		
	\end{proof}
	
	
	The invariance of the coefficients $\tilde{a}_{B}^{AJ}$ under the choice of $\pi$ (as long as $\cDes(\pi) \in B$) also yields a dual structure.
	Defining, in addition to 
	\[
	D_J
	:= \sum_{\substack{\pi \in \symm_n \\ \Des(\pi) = J}} \pi
	\qquad (J \subseteq [n-1]),
	\]
	also
	\[
	cD_A
	:= \sum_{\substack{\pi \in \symm_n \\ \cDes(\pi) \in A}} \pi
	\qquad (A \in c2_{0,n}^{[n]}),
	\]
	the discussion above implies that
	\[
	D_J \cdot cD_A = \sum_{B \in c2_{0,n}^{[n]}} \tilde{a}_{B}^{AJ} cD_B
	\qquad (\forall \,J \subseteq [n-1],\, A \in c2_{0,n}^{[n]}),
	\]
	where multiplication is in the group ring $\ZZ[\symm_n]$.
	This reproves the following result, obtained by Moszkowski~\cite{Mo} and by Aguiar and Petersen~\cite{AP} using different methods.
	
	\begin{corollary}\label{t:left_module}
		The additive free abelian group
		\[
		\mathfrak{cD}_n := \spn_{\,\ZZ} \{cD_A \,:\, A \in c2_{0,n}^{[n]}\}
		\]
		is a left module for Solomon's descent algebra $\mathfrak{D}_n$.
	\end{corollary}
	
	In general, $\mathfrak{cD}_n$ is not a right module of $\mathfrak{D}_n$ and not an algebra.

	
	\section{Open problems and final remarks}\label{sec:open_problems}
	
	
	We close with several remarks and questions.
	
	\medskip
	Recall that one has a tower of ring extensions
	\begin{equation}
		\label{tower-of-rings}
		\Sym_{\QQ}[x_1,\ldots,x_n] \subseteq \cQSym_{\QQ}[x_1,\ldots,x_n] \subseteq \QSym_{\QQ}[x_1,\ldots,x_n] \subseteq \QQ[x_1,\ldots,x_n].
	\end{equation}
	For each ring extension $R \subset S$ within this tower, 
	it is natural to ask about the structure of $S$ as a graded $R$-module.  For example, one might ask whether $S$ is a free graded $R$-module, or one might ask for a minimal generating set of $S$ as a graded $R$-module.
	In this setting of graded $\QQ$-algebras, the
	latter question is equivalent to asking for elements of $S$ whose images in the quotient ring $S/(R_+)$ form a graded $\QQ$-vector space basis, where $(R_+)$ denotes the ideal of $S$ generated by elements of strictly positive degree within $R$.
	
	Garsia and Wallach~\cite{GW} proved that
	the ring $\QSym_\QQ[x_1,\dots,x_n]$ of quasi-symmetric functions in $n$ variables (over $\QQ$) is free over the ring $\Sym_\QQ[x_1,\dots,x_n]$ of symmetric functions in these variables.
	This suggests asking the same question for 
	the leftmost inclusion in \eqref{tower-of-rings}.
	However, one can check that $\cQSym_{\QQ}[x_1,\ldots,x_n]$ is {\it not} a free graded module over $\Sym_{\QQ}[x_1,\ldots,x_n]$ starting at $n=4$, via a Hilbert series calculation:  the quotient of their Hilbert series in $q$ has negative coefficients as a power series, beginning
	$1+2q^6+q^7+2q^8+q^9-q^{10}+\cdots$ .
	
	Aval, Bergeron and Bergeron~\cite{ABB} studied $\QQ[x_1,\ldots,x_n]$ as a module over $\QSym_{\QQ}[x_1,\ldots,x_n]$, which is known to {\it not} be a free module. One might still ask whether $\QQ[x_1,\ldots,x_n]$ happens to be free as a 
	graded module over $\cQSym_{\QQ}[x_1,\ldots,x_n]$, but not surprisingly, this {\it fails} starting at $n=3$; the relevant quotient
	of two Hilbert series has power series expansion that begins $1+2q+2q^2+q^3-q^6-2q^7\cdots$. On
	the other hand, in \cite{ABB} the authors provide a minimal generating set for
	$\QQ[x_1,\ldots,x_n]$ as a $\QSym_{\QQ}[x_1,\ldots,x_n]$-module, with cardinality given by a Catalan number,
	and compute a simple expression for the Hilbert series of the quotient
	$\QQ[x_1,\ldots,x_n] / ( \QSym_{\QQ}[x_1,\ldots,x_n]_+)$.  One might ask analogous questions about the quotient
	$\QQ[x_1,\ldots,x_n]/(\cQSym_{\QQ}[x_1,\ldots,x_n]_+)$, which we have not explored.
	
	Regarding the middle inclusion in \eqref{tower-of-rings}, one can check
	that $\QSym[x_1,\ldots,x_n]$ is {\it not} a free graded module over 
	$\cQSym_{\QQ}[x_1,\ldots,x_n]$ starting at $n=3$; the relevant quotient
	of two Hilbert series has power series expansion that begins $1+q^3+2q^4+2q^5-x^6-x^9-\cdots$.
	We have not explored the dimension and Hilbert series 
	of the quotient $\QSym[x_1,\ldots,x_n]/(\cQSym_{\QQ}[x_1,\ldots,x_n]_+)$.
	
	
	
	\medskip
	
	\newcommand{\alt}{{\operatorname{alt}}}
	
	The same questions can be asked in infinitely many variables. The
	answers, again, are negative. The graded $\QQ$-algebra $\cQSym_\QQ$ is
	not a polynomial ring in graded generators (unlike $\Sym_\QQ$ and
	$\QSym_\QQ$), nor is it a free $\Sym_\QQ$-module; nor is $\QSym_\QQ$ a
	free $\cQSym_\QQ$-module. Unlike in the cases of finitely many
	variables, this does not seem to follow from the Hilbert series alone;
	but it has been checked using {\tt SageMath}~\cite{SageMath}
	by Darij Grinberg, whose arguments we record here. 
	The failure of $\cQSym_\QQ$ to be a free $\Sym_\QQ$-module can be seen as follows: 
	For any nonnegative integer $n$ and any composition $\alpha = \operatorname{cc}_n(J)$, set
	$M^\cyc_{\alpha} := M^\cyc_{n, J}$; furthermore, set $M^\alt_{(a, b, c)} := M^\cyc_{(a, b, c)} - M^\cyc_{(c, b, a)}$ for any composition $(a, b, c)$ of length $3$. 
	Then, it is a pleasant exercise to check that any four positive integers $a, b, c, d$ satisfy
	\[
	p_a M^\alt_{(b, c, d)} - p_b M^\alt_{(a, c, d)} + p_c M^\alt_{(a, b,
		d)} - p_d M^\alt_{(a, b, c)} = 0 ,
	\]
	where $p_1, p_2, p_3, \ldots$ are the power-sum symmetric functions, as usual. 
	Setting $(a, b, c, d) = (1, 2, 3, 4)$, we obtain a
	$\Sym_\QQ$-linear dependence between the four elements
	$M^\alt_{(2, 3, 4)}, M^\alt_{(1, 3, 4)}, M^\alt_{(1, 2, 4)}, M^\alt_{(1, 2, 3)}$ of $\cQSym_\QQ$. 
	Thus, these four elements cannot together belong to a basis of the $\Sym_\QQ$-module $\cQSym_\QQ$. 
	But if $\cQSym_\QQ$ were a free $\Sym_\QQ$-module, then it would have a homogeneous basis (i.e., a basis consisting of homogeneous elements), 
	and moreover such a basis could be found by lifting any homogeneous basis of the quotient vector space 
	$\cQSym_\QQ / \Sym_\QQ^+ \cQSym_\QQ$ (see, e.g., \cite[Prop. 4.11]{RWY},  \cite{Reyes15},\cite[Lemma 1]{Stump}).
	In particular, $\cQSym_\QQ$ would have a basis containing the four elements
	$M^\alt_{(2, 3, 4)}, M^\alt_{(1, 3, 4)}, M^\alt_{(1, 2, 4)}, M^\alt_{(1, 2, 3)}$, because (here is where {\tt SageMath} has been used)
	these elements do not belong to $\Sym_\QQ^+ \cQSym_\QQ$ and thus (due to their distinct degrees) can be extended to a basis of $\cQSym_\QQ /\Sym_\QQ^+ \cQSym_\QQ$. 
	Thus, $\cQSym_\QQ$ is not a free $\Sym_\QQ$-module. 
	
	We claim that this also implies that $\cQSym_\QQ$ is {\it not} a polynomial ring over $\QQ$ with graded
	generators.  
	To see this claim, note that if $\cQSym_\QQ$ were such a polynomial ring, its generators could be chosen as any homogeneous lifts of a $\QQ$-basis for $\cQSym_\QQ/(\cQSym^+_\QQ)^2$.  
	However, this means that the generators for $\cQSym_\QQ$ could be chosen to include all of the power sum symmetric functions $\{p_1,p_2,\ldots\}$:
	the power sums form a subset of a graded polynomial generating set for $\QSym_\QQ$ 
	(see, e.g., \cite[Prop. 6.1.14]{GR}), 
	so their images in $\QSym_\QQ/(\QSym^+_\QQ)^2$ are linearly independent, and hence, {\it a fortiori}, linearly independent in $\cQSym_\QQ/(\cQSym^+_\QQ)^2$.  
	But since $\Sym_\QQ=\QQ[p_1,p_2,\ldots]$, this would make $\cQSym_\QQ$ a free $\Sym_\QQ$-module. 
	
	Knowing that $\cQSym_\QQ$ is not a polynomial ring over $\QQ$ with graded generators can be rephrased as saying that
	it is not isomorphic to the symmetric algebra of any graded $\QQ$-vector space. 
	This, in turn, yields that the graded algebra $\cQSym_\QQ$ does not have the structure of 
	a connected graded $\QQ$-bialgebra\footnote{If $A$ is a connected graded bialgebra over a field, then it is automatically a Hopf algebra.  
		If furthermore $A$ is of finite type over a field of characteristic zero, and has commutative multiplication, then its cocommutative dual $A^*$, will be isomorphic to a universal enveloping algebra by the Cartier--Milnor--Moore Theorem, which then forces $A$ to be a symmetric algebra.}. 
	
	The failure of $\QSym_\QQ$ to be a free
	$\cQSym_\QQ$-module is witnessed by a linear dependence in degree
	$6$.


	
	\medskip
	
	A natural goal is to study specific specializations of cyclic quasi-symmetric functions, e.g., the principal specialization defined by $x_i:=q^{i-1}$ for all $i$.  
	There are possible applications to various statistics of cyclic permutations.
	For example, we are looking for a notion of cyclic major index, which will provide a bivariate analogue of Theorem~\ref{t:cdes_of_cyclic_shuffle}, and cyclic analogues of~\cite[Proposition 7.19.12]{EC2}
	and related identities.
	
	
	
	\medskip
	
	
	A positivity phenomenon, involving cyclic quasi-symmetric functions, was presented in 
	Corollary~\ref{t:Schur_in_hFcyc}.
	It is desired to find more results of this type; here is a conjectured one.
	It was proved in~\cite[Cor. 7.7]{ER3} that, 
	for every integer $0< k< n$,
	the quasi-symmetric function
	\[
	\sum_{\substack{\pi \in \symm_n \\ \cdes(\pi^{-1}) = k}} F_{n,\Des(\pi)}
	\]
	is symmetric and Schur-positive.
	Computational experiments suggest the following (stronger) cyclic version.
	
	\begin{conjecture}
		For every $\varnothing \subsetneq J \subsetneq [n]$ the cyclic quasi-symmetric function
		\[
		\sum_{\substack{\pi \in \symm_n \\ [\cDes(\pi^{-1})] = [J]}} F^\cyc_{n,\cDes(\pi)}
		= \sum_{\substack{\pi\in \symm_n \\ (\exists i)\ \cDes(\pi^{-1}) = J+i}} F^\cyc_{n,\cDes(\pi)}
		\]
		is symmetric and Schur-positive.\footnote{Added in proof: for a most recent affirmative resolution of this conjecture see~\cite{BER}.} 
	\end{conjecture}
	
	
	\medskip
	
	A related problem is the following.
	The proof of Theorem~\ref{t:cyclic_extension} in~\cite{ARR_CDes} is indirect, and involves Postnikov's toric Schur functions. 
	A constructive proof, providing an explicit combinatorial definition of the cyclic descent map, was found very recently by Brice Huang~\cite{Huang}.
	Finding a bijective proof for Theorem~\ref{conj22} is now desired. Such a bijection would provide an effective way to construct subsets $A_\lambda$ and $A_\mu$  as in Proposition~\ref{prop:disconnected_skew1}.
	A generalization of Proposition~\ref{prop:disconnected_skew1} to hook 
	components is further desired.
	A solution 
	would provide a far-reaching generalization of~\cite[Theorem 1]{ER1}.
	
	
	
	
	\medskip
	
	As mentioned in Section~\ref{sec:co},
	it is interesting to define a cyclic analogue of the outer (graded) coproduct on $\QSym$, and study its duality with a suitable version of the product on non-commutative symmetric functions.
	
	\medskip
	
	Finally, cyclic descents were introduced by Cellini~\cite{Cellini} in the search for subalgebras of Solomon's descent algebra. An important subalgebra of the descent algebra is the peak algebra~\cite{peak_paper}.
	
	\begin{problem}
		Define and study cyclic peaks and a cyclic peak algebra.
	\end{problem}
	
	For cyclic peaks in Dyck paths see~\cite[\S 5.1]{AER} .
	
	\section*{Acknowledgements}
	The authors thank Darij Grinberg for many helpful corrections, and for suggesting the
	Hilbert series calculations and further arguments in Section~\ref{sec:open_problems} that answered (negatively) several questions from a previous version of this paper. We also thank
	Yan Zhuang for helpful comments and references.
	

\end{document}